%% file: print.tex
\documentclass[11pt,reqno,a4paper,dvipsnames]{amsart}

\input{preamble_common}
\input{preamble_print}

\begin{document}
\maketitle
\input{abstract}
\tableofcontents
\input{content}
\end{document}

%% file: preamble_common.tex
\title{A blueprint for the Formalization of Seymour's Matroid Decomposition Theorem}

\author[Sergeev et al.]{Ivan Sergeev}
\address{Ivan Sergeev, Institute of Science and Technology Austria}
\email{ivan.sergeev@ist.ac.at}

\author[]{Martin Dvorak}
\address{Martin Dvorak, Institute of Science and Technology Austria}
\email{martin.dvorak@ista.ac.at}

\author[]{Cameron Rampell}
\address{Cameron Rampell}
\email{cameronrampell@gmail.com}

\author[]{Mark Sandey}
\address{Mark Sandey}
\email{mark@sandey-family.com}

\author[]{Pietro Monticone}
\address{Pietro Monticone, University of Trento}
\email{pit.monticone@gmail.com}

\date{\today}

\usepackage{amsmath} 
\usepackage{amssymb} 
\usepackage{amsthm} 
\usepackage[unicode,colorlinks=true,linkcolor=blue,urlcolor=magenta,citecolor=blue]{hyperref}  
\usepackage{tikz}  
\newtheorem{theorem}{Theorem}
\theoremstyle{definition}
\newtheorem{definition}[theorem]{Definition}
\newtheorem{lemma}[theorem]{Lemma}
\newtheorem{corollary}[theorem]{Corollary}

\newcommand{\SeeLean}{See implementation in Lean.}

%% file: preamble_print.tex
\usepackage{expl3}  
\usepackage[margin=2.5cm]{geometry}  
\usepackage{nicematrix} 
\usepackage{todonotes}  

\newcommand{\lean}[1]{}
\newcommand{\discussion}[1]{}
\newcommand{\leanok}{}

\ExplSyntaxOn
\NewDocumentCommand{\uses}{m}
 {\clist_map_inline:nn{#1}{\vphantom{\ref{##1}}}%
  \ignorespaces}
\NewDocumentCommand{\proves}{m}
 {\clist_map_inline:nn{#1}{\vphantom{\ref{##1}}}%
  \ignorespaces}
\ExplSyntaxOff

%% file: abstract.tex
\begin{abstract}
This document is a blueprint for the formalization in Lean of the structural theory of regular matroids underlying Seymour’s decomposition theorem.  
We present a modular account of regularity via totally unimodular representations, show that regularity is preserved under $1$-, $2$-, and $3$-sums, and establish regularity for several special classes of matroids, including graphic, cographic, and the matroid $R_{10}$.  

The blueprint records the logical structure of the proof, the precise dependencies between results, and their correspondence with Lean declarations.  
It is intended both as a guide for the ongoing formalization effort and as a human-readable reference for the organization of the proof.
\end{abstract}

%% file: content.tex
%

\input{introduction}

\input{preliminaries}
\input{sum_1}
\input{sum_2}
\input{sum_3}

\input{special}
\input{cographic}
\input{conclusion}

\input{bibiliography}

%% file: introduction.tex
\section{Introduction}

Seymour's decomposition theorem provides a structural characterization of regular matroids by expressing them as iterated $1$-, $2$-, and $3$-sums of graphic matroids, cographic matroids, and a single exceptional matroid $R_{10}$.  
This result lies at the intersection of matroid theory, linear optimization, and combinatorial geometry, and it plays a central role in the theory of totally unimodular matrices and polynomial-time algorithms. Throughout this blueprint, we primarily work with finite matroids. Several results extend to matroids of finite rank or to infinite matroids, but these generalizations are not pursued systematically here.

Our presentation of the structural theory of regular matroids closely follows the exposition and
terminology of Truemper's monograph~\cite{Truemper}, which serves as a primary reference for the
matroid theory and the matrix-based approach adopted throughout this blueprint. We thank Klaus Truemper for helpful correspondences about the regularity of the 3-sum.

The present document is a \emph{blueprint} for the formalization of this theory in the Lean4 proof assistant.  
Rather than presenting a traditional mathematical exposition, the blueprint records the logical structure of the proof, isolates intermediate results into modular components, and tracks the precise dependencies between statements.  
Each definition, lemma, and theorem is intended to correspond to a Lean declaration, and many proofs are deferred to Lean and indicated as such.

The blueprint is organized into several thematic parts.  
We begin by developing the necessary background on totally unimodular matrices, pivoting operations, and vector matroids.  
We then prove that regularity is preserved under $1$-, $2$-, and $3$-sums of matroids.  
Finally, we establish regularity for certain special matroids -- graphic matroids, cographic matroids, and the matroid $R_{10}$ -- thereby completing the ingredients needed for Seymour's decomposition.

%% file: preliminaries.tex
\section{Preliminaries}

\subsection{Total Unimodularity}

\begin{definition}
    \label{Matrix}
    \leanok
    Matrix is a function that takes a row index and returns a vector, which is a function that takes a column index and returns a value.
    The former aforementioned identity is definitional, the latter is syntactical.
    By abuse of notation $\left(R^Y\right)^X \equiv R^{X \times Y}$ we do not curry functions in this text.
    When a matrix happens to be finite (that is, both $X$ and $Y$ are finite) and its entries are numeric, we like to represent it by a table of numbers.
\end{definition}

\begin{definition}
    \label{Matrix.det}
    \uses{Matrix}
    \leanok
    Let $A$ be a square matrix over a commutative ring whose rows and columns are indexed by the integers $\{1, \dots, n\}$.
    The determinant of $A$ is
    \[
        \det A = \sum_{\sigma \in S_{n}} \left( \operatorname{sgn}(\sigma) \prod_{i = 1}^{n} a_{i, \sigma(i)} \right),
    \]
    where the sum is computed over all permutations $\sigma \in S_{n}$, $\operatorname{sgn}(\sigma)$ denotes the sign of permutation $\sigma$,
    and $a_{i,j} \in R$ is the element of $A$ corresponding to the $i$-th row and the $j$-th column.
\end{definition}

\begin{definition}
    \label{Matrix.IsTotallyUnimodular}
    \uses{Matrix.det}
    \leanok
    Let $R$ be a commutative ring. We say that a matrix $A \in R^{X \times Y}$ is totally unimodular, or TU for short, if for every $k \in \mathbb{N}$, every (not necessarily contiguous) $k \times k$ submatrix $T$ of $A$ has $\det T \in \{0, \pm 1\}$.
\end{definition}

\begin{lemma}
    \label{Matrix.IsTotallyUnimodular.mul_rows}
    \uses{Matrix.IsTotallyUnimodular}
    \leanok
    Let $A$ be a TU matrix. Suppose rows of $A$ are multiplied by $\{0, \pm 1\}$ factors. Then the resulting matrix $A'$ is also TU.
\end{lemma}

\begin{proof}
    \uses{Matrix.IsTotallyUnimodular}
    \leanok
    We prove that $A'$ is TU by Definition~\ref{Matrix.IsTotallyUnimodular}. To this end, let $T'$ be a square submatrix of $A'$. Our goal is to show that $\det T' \in \{0, \pm 1\}$. Let $T$ be the submatrix of $A$ that represents $T'$ before pivoting. If some of the rows of $T$ were multiplied by zeros, then $T'$ contains zero rows, and hence $\det T' = 0$. Otherwise, $T'$ was obtained from $T$ by multiplying certain rows by $-1$. Since $T'$ has finitely many rows, the number of such multiplications is also finite. Since multiplying a row by $-1$ results in the determinant getting multiplied by $-1$, we get $\det T' = \pm \det T \in \{0, \pm 1\}$ as desired.
\end{proof}

\begin{lemma}
    \label{Matrix.IsTotallyUnimodular.mul_cols}
    \uses{Matrix.IsTotallyUnimodular}
    \leanok
    Let $A$ be a TU matrix. Suppose columns of $A$ are multiplied by $\{0, \pm 1\}$ factors. Then the resulting matrix $A'$ is also TU.
\end{lemma}

\begin{proof}
    \uses{Matrix.IsTotallyUnimodular,Matrix.IsTotallyUnimodular.mul_rows}
    \leanok
    Apply Lemma~\ref{Matrix.IsTotallyUnimodular.mul_rows} to $A^{\top}$.
\end{proof}

\begin{definition}
    \label{Matrix.IsPartiallyUnimodular}
    \uses{Matrix.det}
    \leanok
    Given $k \in \mathbb{N}$, we say that a matrix $A$ is $k$-partially unimodular, or $k$-PU for short, if every (not necessarily contiguous, not necessarily injective) $k \times k$ submatrix $T$ of $A$ has $\det T \in \{0, \pm 1\}$.
\end{definition}

\begin{lemma}
    \label{Matrix.isTotallyUnimodular_iff_forall_isPartiallyUnimodular}
    \uses{Matrix.IsTotallyUnimodular,Matrix.IsPartiallyUnimodular}
    \leanok
    A matrix $A$ is TU if and only if $A$ is $k$-PU for every $k \in \mathbb{N}$.
\end{lemma}

\begin{proof}
    \uses{Matrix.IsTotallyUnimodular,Matrix.IsPartiallyUnimodular}
    \leanok
    This follows from Definitions~\ref{Matrix.IsTotallyUnimodular} and~\ref{Matrix.IsPartiallyUnimodular}.
\end{proof}

\subsection{Pivoting}


\begin{definition}
    \label{Matrix.longTableauPivot}
    \uses{Matrix}
    \leanok
    Let $A \in R^{X \times Y}$ be a matrix and let $(x, y) \in X \times Y$ be such that $A (x, y) \neq 0$. A long tableau pivot in $A$ on $(x, y)$ is the operation that maps $A$ to the matrix $A'$ where
    \[
        \forall i \in X, \ \forall j \in Y, \ A' (i, j) = \begin{cases}
            \frac{A (i, j)}{A (x, y)}, & \text{ if } i = x, \\
            A (i, j) - \frac{A (i, y) \cdot A (x, j)}{A (x, y)}, & \text{ if } i \neq x.
        \end{cases}
    \]
\end{definition}

\begin{lemma}
    \label{Matrix.IsTotallyUnimodular.longTableauPivot}
    \uses{Matrix.IsTotallyUnimodular,Matrix.longTableauPivot}
    \leanok
    Let $A \in \mathbb{Q}^{X \times Y}$ be a TU matrix and let $(x, y) \in X \times Y$ be such that $A (x, y) \neq 0$. Then performing the long tableau pivot in $A$ on $(x, y)$ yields a TU matrix.
\end{lemma}

\begin{proof}
    \leanok
    \SeeLean
\end{proof}


\begin{definition}
    \label{Matrix.shortTableauPivot}
    \uses{Matrix.longTableauPivot}
    \leanok
    Let $A \in R^{X \times Y}$ be a matrix and let $(x, y) \in X \times Y$ be such that $A (x, y) \neq 0$. Perform the following sequence of operations.
    \begin{enumerate}
        \item Adjoin the identity matrix $1 \in R^{X \times X}$ to $A$, resulting in the matrix $B = \begin{bmatrix} 1 & A \end{bmatrix} \in R^{X \times (X \oplus Y)}$.
        \item Perform a long tableau pivot in $B$ on $(x, y)$, and let $C$ denote the result.
        \item Swap columns $x$ and $y$ in $C$, and let $D$ be the resulting matrix.
        \item Finally, remove columns indexed by $X$ from $D$, and let $A'$ be the resulting matrix.
    \end{enumerate}
    A short tableau pivot in $A$ on $(x, y)$ is the operation that maps $A$ to the matrix $A'$ defined above.
\end{definition}

\begin{lemma}
    \label{Matrix.shortTableauPivot_eq}
    \uses{Matrix.shortTableauPivot}
    \leanok
    Let $A \in R^{X \times Y}$ be a matrix and let $(x, y) \in X \times Y$ be such that $A (x, y) \neq 0$. Then the short tableau pivot in $A$ on $(x, y)$ maps $A$ to $A'$ with
    \[
        \forall i \in X, \ \forall j \in Y, \ A' (i, j) = \begin{cases}
            \frac{1}{A (x, y)}, & \text{ if } i = x \text{ and } j = y, \\
            \frac{A (x, j)}{A (x, y)}, & \text{ if } i = x \text{ and } j \neq y, \\
            -\frac{A (i, j)}{A (x, y)}, & \text{ if } i \neq x \text{ and } j = y, \\
            A (i, j) - \frac{A (i, y) \cdot A (x, j)}{A (x, y)}, & \text{ if } i \neq x \text{ and } j \neq y.
        \end{cases}
    \]
\end{lemma}

\begin{proof}
    \leanok
    Follows by direct calculation.
\end{proof}

\begin{lemma}
    \label{Matrix.shortTableauPivot_zero}
    \uses{Matrix.shortTableauPivot}
    \leanok
    Let $B = \begin{bmatrix} B_{11} & 0 \\ B_{21} & B_{22} \end{bmatrix} \in \mathbb{Q}^{\{X_{1} \cup X_{2}\} \times \{Y_{1} \times Y_{2}\}}$. Let $B' = \begin{bmatrix} B_{11}' & B_{12}' \\ B_{21}' & B_{22}' \end{bmatrix}$ be the result of performing a short tableau pivot on $(x, y) \in X_{1} \times Y_{1}$ in $B$. Then $B_{12}' = 0$, $B_{22}' = B_{22}$, and $\begin{bmatrix} B_{11}' \\ B_{21}' \end{bmatrix}$ is the matrix resulting from performing a short tableau pivot on $(x, y)$ in $\begin{bmatrix} B_{11} \\ B_{21} \end{bmatrix}$.
\end{lemma}

\begin{proof}
    \leanok
    This follows by a direct calculation. Indeed, because of the $0$ block in $B$, $B_{12}$ and $B_{22}$ remain unchanged, and since $\begin{bmatrix} B_{11} \\ B_{21} \end{bmatrix}$ is a submatrix of $B$ containing the pivot element, performing a short tableau pivot in it is equivalent to performing a short tableau pivot in $B$ and then taking the corresponding submatrix.
\end{proof}

\begin{lemma}
    \label{shortTableauPivot_submatrix_det_abs_eq_div}
    \uses{Matrix.shortTableauPivot}
    \leanok
    Let $k \in \mathbb{N}$, let $A \in \mathbb{Q}^{k \times k}$, and let $A'$ be the result of performing a short tableau pivot in $A$ on $(x, y)$ with $x, y \in \{1, \dots, k\}$ such that $A (x, y) \neq 0$. Then $A'$ contains a submatrix $A''$ of size $(k - 1) \times (k - 1)$ with $|\det A''| = |\det A| / |A (x, y)|$.
\end{lemma}

\begin{proof}
    \leanok
    Let $X = \{1, \dots, k\} \setminus \{x\}$ and $Y = \{1, \dots, k\} \setminus \{y\}$, and let $A'' = A' (X, Y)$. Since $A''$ does not contain the pivot row or the pivot column, $\forall (i, j) \in X \times Y$ we have $A'' (i, j) = A (i, j) - \frac{A (i, y) \cdot A (x, j)}{A (x, y)}$. For $\forall j \in Y$, let $B_{j}$ be the matrix obtained from $A$ by removing row $x$ and column $j$, and let $B_{j}''$ be the matrix obtained from $A''$ by replacing column $j$ with $A (X, y)$ (i.e., the pivot column without the pivot element). The cofactor expansion along row $x$ in $A$ yields
    \[
        \det A = \sum_{j = 1}^{k} (-1)^{y + j} \cdot A (x, j) \cdot \det B_{j}.
    \]
    By reordering columns of every $B_{j}$ to match their order in $B_{j}''$, we get
    \[
        \det A = (-1)^{x + y} \cdot \left( A (x, y) \cdot \det A' - \sum_{j \in Y} A (x, j) \cdot \det B_{j}'' \right).
    \]
    By linearity of the determinant applied to $\det A''$, we have
    \[
        \det A'' = \det A' - \sum_{j \in Y} \frac{A (x, j)}{A (x, y)} \cdot \det B_{j}''
    \]
    Therefore, $|\det A''| = |\det A| / |A (x, y)|$.
\end{proof}

\begin{lemma}
    \label{Matrix.IsTotallyUnimodular.shortTableauPivot}
    \uses{Matrix.IsTotallyUnimodular,Matrix.shortTableauPivot}
    \leanok
    Let $A \in \mathbb{Q}^{X \times Y}$ be a TU matrix and let $(x, y) \in X \times Y$ be such that $A (x, y) \neq 0$. Then performing the short tableau pivot in $A$ on $(x, y)$ yields a TU matrix.
\end{lemma}

\begin{proof}
    \uses{Matrix.IsTotallyUnimodular.longTableauPivot}
    \leanok
    See implementation in Lean, which uses Lemma \ref{Matrix.IsTotallyUnimodular.longTableauPivot}.
\end{proof}

\subsection{Vector Matroids}

\begin{definition}
    \label{Matroid}
    \leanok
    A matroid $M$ is a pair $(E, \mathcal{I})$ where $E$ is a (possibly infinite) set and $\mathcal I \in 2^E$ is such that:
    \begin{enumerate}
        \item $\emptyset \in \mathcal I$
        \item If $I \in \mathcal I$ and $J \subseteq I$, then $I \in \mathcal I$.
        \item If $I \in \mathcal I$ is not maximal (with respect to set inclusion) and $B \in \mathcal I$ is maximal,
            then there exists an $x \in B \setminus I$ such that $I \cup \{x\} \in \mathcal{I}$.
        \item If $X \subseteq E$ and $I \subseteq X$ is such that $I \in \mathcal I$, then there exists an $J \in \mathcal I$ with $I \subseteq J \subseteq X$
            that is maximal with respect to set inclusion.
    \end{enumerate}
    We call $E$ the ground set of $M$ and $\mathcal{I}$ the collection of independent sets in $M$. We say that $B \in \mathcal I$ is a base of $M$ if $B$ is maximal in $\mathcal I$.
\end{definition}

\begin{definition}
    \label{VectorMatroid}
    \uses{Matrix, Matroid}
    \leanok
    Let $R$ be a division ring, let $X$ and $Y$ be sets, and let $A \in R^{X \times Y}$ be a matrix. The vector matroid of $A$ is the matroid $M = (Y, \mathcal{I})$ where a set $I \subset Y$ is independent in $M$ if and only if the columns of $A$ indexed by $I$ are linearly independent.
\end{definition}

\begin{definition}
    \label{StandardRepr}
    \uses{VectorMatroid}
    \leanok
    Let $R$ be a division ring, let $X$ and $Y$ be disjoint sets, and let $S \in R^{X \times Y}$ be a matrix. Let $A = \begin{bmatrix} 1 & S \end{bmatrix} \in R^{X \times (X \cup Y)}$ be the matrix obtained from $S$ by adjoining the identity matrix as columns, and let $M$ be the vector matroid of $A$. Then $S$ is called the standard representation of $M$.
\end{definition}

\begin{lemma}
    \label{StandardRepr.toMatroid_isBase_X}
    \uses{StandardRepr}
    \leanok
    Let $S \in R^{X \times Y}$ be a standard representation of a vector matroid $M$. Then $X$ is a base in $M$.
\end{lemma}

\begin{proof}
    \leanok
    \SeeLean
\end{proof}




\begin{lemma}
    \label{Matrix.fromRows_zero_reindex_toMatroid}
    \uses{VectorMatroid}
    \leanok
    Adding extra zero rows to a full representation matrix of a vector matroid does not change the matroid.
\end{lemma}

\begin{proof}
    \leanok
    \SeeLean
\end{proof}

\begin{lemma}
    \label{VectorMatroid.exists_standardRepr_isBase_isTotallyUnimodular}
    \uses{Matrix.IsTotallyUnimodular,VectorMatroid,StandardRepr}
    \leanok
    Let $A \in \mathbb{Q}^{X \times Y}$ be a TU matrix, let $M$ be the vector matroid of $A$, and let $B$ be a base of $M$. Then there exists a matrix $S \in \mathbb{Q}^{B \times (Y \setminus B)}$ such that $S$ is TU and $S$ is a standard representation of $M$.
\end{lemma}

\begin{proof}
    \uses{Matrix.IsTotallyUnimodular.longTableauPivot,Matrix.fromRows_zero_reindex_toMatroid}
    \leanok
    See Lean implementation, which uses Lemmas \ref{Matrix.IsTotallyUnimodular.longTableauPivot} and \ref{Matrix.fromRows_zero_reindex_toMatroid}.
\end{proof}


\begin{definition}
    \label{Matrix.support}
    \uses{Matrix}
    \leanok
    Let $R$ be a magma containing zero. The support of matrix $A \in R^{X \times Y}$ is $A^{\#} \in \{0, 1\}^{X \times Y}$ given by
    \[
        \forall i \in X, \ \forall j \in Y, \ A^{\#} (i, j) = \begin{cases}
            0, & \text{ if } A (i, j) = 0, \\
            1, & \text{ if } A (i, j) \neq 0.
        \end{cases}
    \]
\end{definition}

\begin{lemma}
    \label{Matrix.support_transpose}
    \uses{Matrix.support}
    \leanok
    Transpose of a support matrix is equal to a support of the transposed matrix.
\end{lemma}

\begin{proof}
    \leanok
    Definitional equality.
\end{proof}

\begin{lemma}
    \label{Matrix.support_submatrix}
    \uses{Matrix.support}
    \leanok
    Submatrix of a support matrix is equal to a support matrix of the submatrix.
\end{lemma}

\begin{proof}
    \leanok
    Definitional equality.
\end{proof}

\begin{lemma}
    \label{Matrix.support_Z2}
    \uses{Matrix.support}
    \leanok
    If $A$ is a matrix over $\mathbb{Z}_{2}$, then $A^{\#} = A$.
\end{lemma}

\begin{proof}
    \leanok
    Check elementwise equality.
\end{proof}

\begin{lemma}
    \label{support_eq_support_of_same_matroid_same_X}
    \uses{Matrix.support,StandardRepr}
    \leanok
    If two standard representation matrices of the same matroid have the same base, then they have the same support.
\end{lemma}

\begin{proof}
    \leanok
    \SeeLean
\end{proof}

\begin{lemma}
    \label{Matrix.isUnit_iff_isUnit_det}
    \uses{Matrix.det}
    \leanok
    A square matrix is invertible iff its determinant is invertible.
\end{lemma}

\begin{proof}
    \leanok
    This result is proved in Mathlib.
\end{proof}

\begin{lemma}
    \label{Matrix.IsTotallyUnimodular.linearIndependent_iff_support_linearIndependent_of_finite_of_finite}
    \uses{Matrix.IsTotallyUnimodular,Matrix.support}
    \leanok
    Let $A$ be a rational TU matrix with finite number of rows and finite number of columns.
    Its rows are linearly independent iff the rows of its support matrix are linearly independent.
\end{lemma}

\begin{proof}
    \uses{Matrix.support_submatrix,Matrix.isUnit_iff_isUnit_det}
    \leanok
    See Lean implementation, which uses Lemmas \ref{Matrix.support_submatrix} and \ref{Matrix.isUnit_iff_isUnit_det}.
\end{proof}

\begin{lemma}
    \label{Matrix.IsTotallyUnimodular.linearIndependent_iff_support_linearIndependent_of_finite}
    \uses{Matrix.IsTotallyUnimodular,Matrix.support}
    \leanok
    Let $A$ be a rational TU matrix with finite number of rows.
    Its rows are linearly independent iff the rows of its support matrix are linearly independent.
\end{lemma}

\begin{proof}
    \uses{Matrix.IsTotallyUnimodular.linearIndependent_iff_support_linearIndependent_of_finite_of_finite}
    \leanok
    See Lean implementation, which uses Lemma \ref{Matrix.IsTotallyUnimodular.linearIndependent_iff_support_linearIndependent_of_finite_of_finite}.
\end{proof}

\begin{lemma}
    \label{Matrix.IsTotallyUnimodular.linearIndependent_iff_support_linearIndependent}
    \uses{Matrix.IsTotallyUnimodular,Matrix.support}
    \leanok
    Let $A$ be a rational TU matrix.
    Its rows are linearly independent iff the rows of its support matrix are linearly independent.
\end{lemma}

\begin{proof}
    \uses{Matrix.IsTotallyUnimodular.linearIndependent_iff_support_linearIndependent_of_finite}
    \leanok
    See Lean implementation, which uses Lemma \ref{Matrix.IsTotallyUnimodular.linearIndependent_iff_support_linearIndependent_of_finite}.
\end{proof}

\begin{lemma}
    \label{Matrix.IsTotallyUnimodular.toMatroid_eq_support_toMatroid}
    \uses{Matrix.IsTotallyUnimodular,Matrix.support,StandardRepr}
    \leanok
    Let $A$ be a TU matrix.
    \begin{enumerate}
        \item If a matroid is represented by $A$, then it is also represented by $A^{\#}$.
        \item If a matroid is represented by $A^{\#}$, then it is also represented by $A$.
    \end{enumerate}
\end{lemma}

\begin{proof}
    \uses{Matrix.support_transpose,Matrix.support_submatrix,Matrix.IsTotallyUnimodular.linearIndependent_iff_support_linearIndependent}
    \leanok
    See Lean implementation, which uses Lemmas \ref{Matrix.support_transpose}, \ref{Matrix.support_submatrix}, and \ref{Matrix.IsTotallyUnimodular.linearIndependent_iff_support_linearIndependent}.
\end{proof}

\subsection{Regular Matroids}

\begin{definition}
    \label{Matroid.IsRegular}
    \uses{Matroid,VectorMatroid,Matrix.IsTotallyUnimodular}
    \leanok
    A matroid $M$ is regular if there exists a TU matrix $A \in \mathbb{Q}^{X \times Y}$ such that $M$ is a vector matroid of $A$.
\end{definition}

\begin{definition}
    \label{Matrix.IsTuSigningOf}
    \uses{Matrix.IsTotallyUnimodular}
    \leanok
    We say that $A' \in \mathbb{Q}^{X \times Y}$ is a TU signing of $A \in \mathbb{Z}_{2}^{X \times Y}$ if $A'$ is TU and
    \[
        \forall i \in X, \ \forall j \in Y, \ |A' (i, j)| = A (i, j).
    \]
\end{definition}


\begin{lemma}
    \label{StandardRepr.toMatroid_isRegular_iff_hasTuSigning}
    \uses{StandardRepr,Matroid.IsRegular,Matrix.IsTuSigningOf}
    \leanok
    Let $B \in \mathbb{Z}_{2}^{X \times Y}$ be a standard representation matrix of a matroid $M$. Then $M$ is regular if and only if $B$ has a TU signing.
\end{lemma}

\begin{proof}
    \uses{Matroid.IsRegular,Matrix.IsTuSigningOf,StandardRepr.toMatroid_isBase_X,VectorMatroid.exists_standardRepr_isBase_isTotallyUnimodular,support_eq_support_of_same_matroid_same_X,Matrix.IsTotallyUnimodular.toMatroid_eq_support_toMatroid, Matrix.support_Z2}
    \leanok
    Suppose that $M$ is regular. By Definition~\ref{Matroid.IsRegular}, there exists a TU matrix $A \in \mathbb{Q}^{X \times Y}$ such that $M$ is a vector matroid of $A$. By Lemma~\ref{StandardRepr.toMatroid_isBase_X}, $X$ (the row set of $B$) is a base of $M$. By Lemma~\ref{VectorMatroid.exists_standardRepr_isBase_isTotallyUnimodular}, $A$ can be converted into a standard representation matrix $B' \in \mathbb{Q}^{X \times Y}$ of $M$ such that $B'$ is also TU. Since $B'$ and $B$ are both standard representations of $M$, by Lemma~\ref{support_eq_support_of_same_matroid_same_X} the support matrices $(B')^{\#}$ and $B^{\#}$ are the same. Lemma \ref{Matrix.support_Z2} gives $B^{\#} = B$. Thus, $B'$ is TU and $(B')^{\#} = B$, so $B'$ is a TU signing of $B$.

    Suppose that $B$ has a TU signing $B' \in \mathbb{Q}^{X \times Y}$. Then $A = [1 \mid B']$ is TU, as it is obtained from $B'$ by adjoining the identity matrix. Moreover, by Lemma~\ref{Matrix.IsTotallyUnimodular.toMatroid_eq_support_toMatroid}, $A$ represents the same matroid as $A^{\#} = [1 \mid B]$, which is $M$. Thus, $A$ is a TU matrix representing $M$, so $M$ is regular.
\end{proof}

%% file: sum_1.tex
\section{Regularity of 1-Sum}

\begin{definition}
    \label{standardReprSum1}
    \uses{StandardRepr}
    \leanok
    Let $R$ be a magma containing zero (we will use $R = \mathbb{Z}_{2}$ and $R = \mathbb{Q}$). Let $B_{\ell} \in R^{X_{\ell} \times Y_{\ell}}$ and $B_{r} \in R^{X_{r} \times Y_{r}}$ be matrices where $X_{\ell}, Y_{\ell}, X_{r}, Y_{r}$ are pairwise disjoint sets. The $1$-sum $B = B_{\ell} \oplus_{1} B_{r}$ of $B_{\ell}$ and $B_{r}$ is
    \[
        B = \begin{bmatrix} B_{\ell} & 0 \\ 0 & B_{r} \end{bmatrix} \in R^{(X_{\ell} \cup X_{r}) \times (Y_{\ell} \cup Y_{r})}.
    \]
\end{definition}

\begin{definition}
    \label{Matroid.IsSum1of}
    \uses{Matroid,StandardRepr,standardReprSum1}
    \leanok
    A matroid $M$ is a $1$-sum of matroids $M_{\ell}$ and $M_{r}$ if there exist standard $\mathbb{Z}_{2}$ representation matrices $B_{\ell}$, $B_{r}$, and $B$ (for $M_{\ell}$, $M_{r}$, and $M$, respectively) of the form given in Definition~\ref{standardReprSum1}.
\end{definition}

\begin{lemma}
    \label{Matrix.det_fromBlocks_zero}
    \uses{Matrix.det}
    \leanok
    Let $A$ be a square matrix of the form $A = \begin{bmatrix} A_{11} & A_{12} \\ 0 & A_{22} \end{bmatrix}$. Then $\det A = \det A_{11} \cdot \det A_{22}$.
\end{lemma}

\begin{proof}
    \uses{Matrix.det}
    \leanok
    This result is proved in Mathlib.
\end{proof}

\begin{lemma}
    \label{Matrix.fromBlocks_isTotallyUnimodular}
    \uses{standardReprSum1,Matrix.IsTotallyUnimodular}
    \leanok
    Let $B_{\ell}$ and $B_{r}$ from Definition~\ref{standardReprSum1} be TU matrices (over $\mathbb{Q}$). Then $B = B_{\ell} \oplus_{1} B_{r}$ is TU.
\end{lemma}

\begin{proof}
    \uses{standardReprSum1,Matrix.IsTotallyUnimodular,Matrix.det_fromBlocks_zero}
    \leanok
    We prove that $B$ is TU by Definition~\ref{Matrix.IsTotallyUnimodular}. To this end, let $T$ be a square submatrix of $B$. Our goal is to show that $\det T \in \{0, \pm 1\}$.

    Let $T_{\ell}$ and $T_{r}$ denote the submatrices in the intersection of $T$ with $B_{\ell}$ and $B_{r}$, respectively. Then $T$ has the form
    \[
        T = \begin{bmatrix} T_{\ell} & 0 \\ 0 & T_{r} \end{bmatrix}.
    \]

    First, suppose that $T_{\ell}$ and $T_{r}$ are square. Then $\det T = \det T_{\ell} \cdot \det T_{r}$ by Lemma~\ref{Matrix.det_fromBlocks_zero}. Moreover, $\det T_{\ell}, \det T_{r} \in \{0, \pm 1\}$, since $T_{\ell}$ and $T_{r}$ are square submatrices of TU matrices $B_{\ell}$ and $B_{r}$, respectively. Thus, $\det T \in \{0, \pm 1\}$, as desired.

    Without loss of generality we may assume that $T_{\ell}$ has fewer rows than columns. Otherwise we can transpose all matrices and use the same proof, since TUness and determinants are preserved under transposition. Thus, $T$ can be represented in the form
    \[
        T = \begin{bmatrix} T_{11} & T_{12} \\ 0 & T_{22} \end{bmatrix},
    \]
    where $T_{11}$ contains $T_{\ell}$ and some zero rows, $T_{22}$ is a submatrix of $T_{r}$, and $T_{12}$ contains the rest of the rows of $T_{r}$ (not contained in $T_{22}$) and some zero rows. By Lemma~\ref{Matrix.det_fromBlocks_zero}, we have $\det T = \det T_{11} \cdot \det T_{22}$. Since $T_{11}$ contains at least one zero row, $\det T_{11} = 0$. Thus, $\det T = 0 \in \{0, \pm 1\}$, as desired.
\end{proof}

\begin{theorem}
    \label{Matroid.IsSum1of.isRegular}
    \uses{Matroid.IsSum1of,Matroid.IsRegular}
    \leanok
    Let $M$ be a $1$-sum of regular matroids $M_{\ell}$ and $M_{r}$. Then $M$ is also regular.
\end{theorem}

\begin{proof}
    \uses{StandardRepr,Matroid.IsSum1of,Matroid.IsRegular,StandardRepr.toMatroid_isRegular_iff_hasTuSigning,Matrix.fromBlocks_isTotallyUnimodular,Matrix.IsTuSigningOf}
    \leanok
    Let $B_{\ell}$, $B_{r}$, and $B$ be standard $\mathbb{Z}_{2}$ representation matrices from Definition~\ref{Matroid.IsSum1of}. Since $M_{\ell}$ and $M_{r}$ are regular, by Lemma~\ref{StandardRepr.toMatroid_isRegular_iff_hasTuSigning}, $B_{\ell}$ and $B_{r}$ have TU signings $B_{\ell}'$ and $B_{r}'$, respectively. Then $B' = B_{\ell}' \oplus_{1} B_{r}'$ is a TU signing of $B$. Indeed, $B'$ is TU by Lemma~\ref{Matrix.fromBlocks_isTotallyUnimodular}, and a direct calculation shows that $B'$ is a signing of $B$. Thus, $M$ is regular by Lemma~\ref{StandardRepr.toMatroid_isRegular_iff_hasTuSigning}.
\end{proof}

%% file: sum_2.tex
\section{Regularity of 2-Sum}

\begin{definition}
    \label{standardReprSum2}
    \uses{StandardRepr}
    \leanok
    Let $R$ be a semiring (we will use $R = \mathbb{Z}_{2}$ and $R = \mathbb{Q}$). Let $B_{\ell} \in R^{X_{\ell} \times Y_{\ell}}$ and $B_{r} \in R^{X_{r} \times Y_{r}}$ where $X_{\ell} \cap X_{r} = \{x\}$, $Y_{\ell} \cap Y_{r} = \{y\}$, $X_{\ell}$ is disjoint with $Y_{\ell}$ and $Y_{r}$, and $X_{r}$ is disjoint with $Y_{\ell}$ and $Y_{r}$. Additionally, let $A_{\ell} = B_{\ell} (X_{\ell} \setminus \{x\}, Y_{\ell})$ and $A_{r} = B_{r} (X_{r}, Y_{r} \setminus \{y\})$, and suppose $r = B_{\ell} (x, Y_{\ell}) \neq 0$ and $c = B_{r} (X_{r}, y) \neq 0$. Then the $2$-sum $B = B_{\ell} \oplus_{2, x, y} B_{r}$ of $B_{\ell}$ and $B_{r}$ is defined as
    \[
        B = \begin{bmatrix} A_{\ell} & 0 \\ D & A_{r} \end{bmatrix}
        \quad \text{where} \quad
        D = c \otimes r.
    \]
    Here $D \in R^{X_{r} \times Y_{\ell}}$, and the indexing is consistent everywhere.
\end{definition}

\begin{definition}
    \label{Matroid.IsSum2of}
    \uses{Matroid,StandardRepr,standardReprSum2}
    \leanok
    A matroid $M$ is a $2$-sum of matroids $M_{\ell}$ and $M_{r}$ if there exist standard $\mathbb{Z}_{2}$ representation matrices $B_{\ell}$, $B_{r}$, and $B$ (for $M_{\ell}$, $M_{r}$, and $M$, respectively) of the form given in Definition~\ref{standardReprSum2}.
\end{definition}

\begin{lemma}
    \label{Matrix.IsTotallyUnimodular.fromCols_outer}
    \uses{standardReprSum2,Matrix.IsTotallyUnimodular}
    \leanok
    Let $B_{\ell}$ and $B_{r}$ from Definition~\ref{standardReprSum2} be TU matrices (over $\mathbb{Q}$). Then $C = \begin{bmatrix} D & A_{r} \end{bmatrix}$ is TU.
\end{lemma}

\begin{proof}
    \uses{standardReprSum2,Matrix.IsTotallyUnimodular,Matrix.IsTotallyUnimodular.mul_cols}
    \leanok
    Since $B_{\ell}$ is TU, all its entries are in $\{0, \pm 1\}$. In particular, $r$ is a $\{0, \pm 1\}$ vector. Therefore, every column of $D$ is a copy of $y$, $-y$, or the zero column. Thus, $C$ can be obtained from $B_{r}$ by adjoining zero columns, duplicating the $y$ column, and multiplying some columns by $-1$. Since all these operations preserve TUess and since $B_{r}$ is TU, $C$ is also TU.
\end{proof}

\begin{lemma}
    \label{standardReprSum2_shortTableauPivot}
    \uses{standardReprSum2,Matrix.shortTableauPivot}
    \leanok
    Let $B_{\ell}$ and $B_{r}$ be matrices from Definition~\ref{standardReprSum2}. Let $B_{\ell}'$ and $B'$ be the matrices obtained by performing a short tableau pivot on $(x_{\ell}, y_{\ell}) \in X_{\ell} \times Y_{\ell}$ in $B_{\ell}$ and $B$, respectively. Then $B' = B_{\ell}' \oplus_{2, x, y} B_{r}$.
\end{lemma}

\begin{proof}
    \uses{standardReprSum2,Matrix.shortTableauPivot,Matrix.shortTableauPivot_zero,Matrix.shortTableauPivot_eq}
    \leanok
    Let
    \[
        B_{\ell}' = \begin{bmatrix} A_{\ell}' \\ r' \end{bmatrix}, \quad
        B' = \begin{bmatrix} B_{11}' & B_{12}' \\ B_{21}' & B_{22}' \end{bmatrix}
    \]
    where the blocks have the same dimensions as in $B_{\ell}$ and $B$, respectively. By Lemma~\ref{Matrix.shortTableauPivot_zero}, $B_{11}' = A_{\ell}'$, $B_{12}' = 0$, and $B_{22}' = A_{r}$. Equality $B_{21}' = c \otimes r'$ can be verified via a direct calculation. Thus, $B' = B_{\ell}' \oplus_{2, x, y} B_{r}$.
\end{proof}

\begin{lemma}
    \label{standardReprSum2_isTotallyUnimodular}
    \uses{standardReprSum2,Matrix.IsTotallyUnimodular}
    \leanok
    Let $B_{\ell}$ and $B_{r}$ from Definition~\ref{standardReprSum2} be TU matrices (over $\mathbb{Q}$). Then $B_{\ell} \oplus_{2, x, y} B_{r}$ is TU.
\end{lemma}

\begin{proof}
    \uses{Matrix.isTotallyUnimodular_iff_forall_isPartiallyUnimodular,standardReprSum2,Matrix.IsTotallyUnimodular.fromCols_outer,shortTableauPivot_submatrix_det_abs_eq_div,standardReprSum2_shortTableauPivot,Matrix.IsTotallyUnimodular.shortTableauPivot}
    \leanok
    By Lemma~\ref{Matrix.isTotallyUnimodular_iff_forall_isPartiallyUnimodular}, it suffices to show that $B_{\ell} \oplus_{2, x, y} B_{r}$ is $k$-PU for every $k \in \mathbb{N}$. We prove this claim by induction on $k$. The base case with $k = 1$ holds, since all entries of $B_{\ell} \oplus_{2, x, y} B_{r}$ are in $\{0, \pm 1\}$ by construction.

    Suppose that for some $k \in \mathbb{N}$ we know that for any TU matrices $B_{\ell}'$ and $B_{r}'$ (from Definition~\ref{standardReprSum2}) their $2$-sum $B_{\ell}' \oplus_{2, x, y} B_{r}'$ is $k$-PU. Now, given TU matrices $B_{\ell}$ and $B_{r}$ (from Definition~\ref{standardReprSum2}), our goal is to show that $B = B_{\ell} \oplus_{2, x, y} B_{r}$ is $(k + 1)$-PU, i.e., that every $(k + 1) \times (k + 1)$ submatrix $T$ of $B$ has $\det T \in \{0, \pm 1\}$.

    First, suppose that $T$ has no rows in $X_{\ell}$. Then $T$ is a submatrix of $\begin{bmatrix} D & A_{r} \end{bmatrix}$, which is TU by Lemma~\ref{Matrix.IsTotallyUnimodular.fromCols_outer}, so $\det T \in \{0, \pm 1\}$. Thus, we may assume that $T$ contains a row $x_{\ell} \in X_{\ell}$.

    Next, note that without loss of generality we may assume that there exists $y_{\ell} \in Y_{\ell}$ such that $T (x_{\ell}, y_{\ell}) \neq 0$. Indeed, if $T (x_{\ell}, y) = 0$ for all $y$, then $\det T = 0$ and we are done, and $T (x_{\ell}, y) = 0$ holds whenever $y \in Y_{r}$.

    Since $B$ is $1$-PU, all entries of $T$ are in $\{0, \pm 1\}$, and hence $T (x_{\ell}, y_{\ell}) \in \{\pm 1\}$. Thus, by Lemma~\ref{shortTableauPivot_submatrix_det_abs_eq_div}, performing a short tableau pivot in $T$ on $(x_{\ell}, y_{\ell})$ yields a matrix that contains a $k \times k$ submatrix $T''$ such that $|\det T| = |\det T''|$. Since $T$ is a submatrix of $B$, matrix $T''$ is a submatrix of the matrix $B'$ resulting from performing a short tableau pivot in $B$ on the same entry $(x_{\ell}, y_{\ell})$. By Lemma~\ref{standardReprSum2_shortTableauPivot}, we have $B' = B_{\ell}' \oplus_{2, x, y} B_{r}$ where $B_{\ell}'$ is the result of performing a short tableau pivot in $B_{\ell}$ on $(x_{\ell}, y_{\ell})$. Since $B_{\ell}$ is TU, by Lemma \ref{Matrix.IsTotallyUnimodular.shortTableauPivot}, $B_{\ell}'$ is also TU. Thus, by the inductive hypothesis applied to $T''$ and $B_{\ell}' \oplus_{2, x, y} B_{r}$, we have $\det T'' \in \{0, \pm 1\}$. Since $|\det T| = |\det T''|$, we conclude that $\det T \in \{0, \pm 1\}$.
\end{proof}

\begin{theorem}
    \label{Matroid.IsSum2of.isRegular}
    \uses{Matroid.IsRegular,Matroid.IsSum2of}
    \leanok
    Let $M$ be a $2$-sum of regular matroids $M_{\ell}$ and $M_{r}$. Then $M$ is also regular.
\end{theorem}

\begin{proof}
    \uses{StandardRepr,Matroid.IsSum2of,StandardRepr.toMatroid_isRegular_iff_hasTuSigning,standardReprSum2_isTotallyUnimodular,Matrix.IsTuSigningOf}
    \leanok
    Let $B_{\ell}$, $B_{r}$, and $B$ be standard $\mathbb{Z}_{2}$ representation matrices from Definition~\ref{Matroid.IsSum2of}. Since $M_{\ell}$ and $M_{r}$ are regular, by Lemma~\ref{StandardRepr.toMatroid_isRegular_iff_hasTuSigning}, $B_{\ell}$ and $B_{r}$ have TU signings $B_{\ell}'$ and $B_{r}'$, respectively. Then $B' = B_{\ell}' \oplus_{2, x, y} B_{r}'$ is a TU signing of $B$. Indeed, $B'$ is TU by Lemma~\ref{standardReprSum2_isTotallyUnimodular}, and a direct calculation verifies that $B'$ is a signing of $B$. Thus, $M$ is regular by Lemma~\ref{StandardRepr.toMatroid_isRegular_iff_hasTuSigning}.
\end{proof}

%% file: sum_3.tex
\section{Regularity of 3-Sum}

\subsection{Definition}

\begin{definition}
    \label{standardReprSum3}
    \uses{StandardRepr}
    \leanok
    Let $X_{\ell}$, $Y_{\ell}$, $X_{r}$, and $Y_{r}$ be sets satisfying the following properties:
    \begin{itemize}
        \item $X_{\ell} \cap X_{r} = \{x_{2}, x_{1}, x_{0}\}$ for some distinct $x_{0}$, $x_{1}$, and $x_{2}$;
        \item $Y_{\ell} \cap Y_{r} = \{y_{0}, y_{1}, y_{2}\}$ for some distinct $y_{0}$, $y_{1}$, and $y_{2}$;
        \item $X_{\ell}$ is disjoint with $Y_{r}$; and
        \item $Y_{\ell}$ is disjoint with $X_{r}$.
    \end{itemize}
    Let $B_{\ell} \in \mathbb{Z}_{2}^{X_{\ell} \times Y_{\ell}}$ and $B_{r} \in \mathbb{Z}_{2}^{X_{r} \times Y_{r}}$ be matrices of the form
    \begin{center}
        \noindent
        \begin{tikzpicture}
            \begin{scope}[scale=0.5, shift={(-5.5, -3)}]
                \node[anchor=east] at (0, 3) {$B_{\ell} =$};
                \draw (0, 0) -- (5, 0) -- (5, 6) -- (0, 6) -- cycle;
                \draw (0, 2) -- (5, 2);
                \draw (4, 0) -- (4, 6);
                \draw (2, 0) -- (2, 3) -- (5, 3);
                \draw (4, 1) -- (5, 1);
                \draw (3, 2) -- (3, 3);
                \node at (2, 4) {$A_{\ell}$};
                \node at (1, 1) {$D_{\ell}$};
                \node at (4.5, 4) {$0$};
                \node at (2.5, 2.5) {$1$};
                \node at (3.5, 2.5) {$1$};
                \node at (4.5, 2.5) {$0$};
                \node at (4.5, 1.5) {$1$};
                \node at (4.5, 0.5) {$1$};
                \node at (3, 1) {$D_{0}$};
            \end{scope}
            \node[anchor=west] at (0, 0) {and};
            \begin{scope}[scale=0.5, shift={(3.5, -2.5)}]
                \node[anchor=east] at (0, 2.5) {$B_{r} =$};
                \draw (0, 0) -- (6, 0) -- (6, 5) -- (0, 5) -- cycle;
                \draw (2, 0) -- (2, 5);
                \draw (0, 4) -- (6, 4);
                \draw (0, 2) -- (3, 2) -- (3, 5);
                \draw (1, 4) -- (1, 5);
                \draw (2, 3) -- (3, 3);
                \node at (1, 1) {$D_{r}$};
                \node at (1, 3) {$D_{0}$};
                \node at (0.5, 4.5) {$1$};
                \node at (1.5, 4.5) {$1$};
                \node at (4, 4.5) {$0$};
                \node at (2.5, 2.5) {$1$};
                \node at (2.5, 3.5) {$1$};
                \node at (2.5, 4.5) {$0$};
                \node at (4, 2) {$A_{r}$};
            \end{scope}
        \end{tikzpicture}
    \end{center}
    where $D_{0}$ is invertible. Then the $3$-sum $B = B_{\ell} \oplus_{3} B_{r}$ of $B_{\ell}$ and $B_{r}$ is defined as
    \begin{center}
        \noindent
        \begin{tikzpicture}
            \begin{scope}[scale=0.5, shift={(0, -4)}]
                \node[anchor=east] at (0, 4) {$B =$};
                \draw (0, 0) -- (8, 0) -- (8, 8) -- (0, 8) -- cycle;
                \draw (0, 4) -- (8, 4);
                \draw (4, 0) -- (4, 8);
                \draw (0, 2) -- (5, 2) -- (5, 5) -- (2, 5) -- (2, 0);
                \draw (3, 4) -- (3, 5);
                \draw (4, 3) -- (5, 3);
                \node at (2, 6) {$A_{\ell}$};
                \node at (6, 2) {$A_{r}$};
                \node at (6, 6) {$0$};
                \node at (1, 1) {$D_{\ell r}$};
                \node at (1, 3) {$D_{\ell}$};
                \node at (3, 1) {$D_{r}$};
                \node at (3, 3) {$D_{0}$};
                \node at (2.5, 4.5) {$1$};
                \node at (3.5, 4.5) {$1$};
                \node at (4.5, 4.5) {$0$};
                \node at (4.5, 3.5) {$1$};
                \node at (4.5, 2.5) {$1$};
            \end{scope}
            \node[anchor=west] at (4.25, 0) {where $D_{\ell r} = D_{r} \cdot (D_{0})^{-1} \cdot D_{\ell}$.};
        \end{tikzpicture}
    \end{center}
    Here the indexing is consistent between all the matrices, $D_{0} \in \mathbb{Z}_{2}^{\{x_{1}, x_{0}\} \times \{y_{0}, y_{1}\}}$, and the submatrix
    \begin{tikzpicture}
        \begin{scope}[scale=0.5, shift={(-3.5, -1.5)}]
            \draw (0, 0) -- (3, 0) -- (3, 3) -- (0, 3) -- cycle;
            \draw (0, 2) -- (3, 2);
            \draw (2, 0) -- (2, 3);
            \draw (1, 2) -- (1, 3);
            \draw (2, 1) -- (3, 1);
            \node at (1, 1) {$D_{0}$};
            \node at (0.5, 2.5) {$1$};
            \node at (1.5, 2.5) {$1$};
            \node at (2.5, 2.5) {$0$};
            \node at (2.5, 1.5) {$1$};
            \node at (2.5, 0.5) {$1$};
        \end{scope}
    \end{tikzpicture}
    is indexed by $\{x_{2}, x_{1}, x_{0}\} \times \{y_{0}, y_{1}, y_{2}\}$ in $B_{\ell}$, $B_{r}$, and $B$.
\end{definition}

\begin{definition}
    \label{Matroid.IsSum3of}
    \uses{Matroid,StandardRepr,standardReprSum3}
    \leanok
    A matroid $M$ is a $3$-sum of matroids $M_{\ell}$ and $M_{r}$ if there exist standard $\mathbb{Z}_{2}$ representation matrices $B_{\ell}$, $B_{r}$, and $B$ (for $M_{\ell}$, $M_{r}$, and $M$, respectively) of the form given in Definition~\ref{standardReprSum3}.
\end{definition}

\subsection{Canonical Signing}

\begin{lemma}
    \label{Matrix.isUnit_2x2}
    \uses{Matrix}
    \leanok
    Let $D_{0} \in \mathbb{Z}_{2}^{\{x_{1}, x_{0}\} \times \{y_{0}, y_{1}\}}$ be an invertible matrix. Then, up to reindexing of rows and columns, either $D_{0} = \begin{bmatrix} 1 & 0 \\ 0 & 1 \\ \end{bmatrix}$ or $D_{0} = \begin{bmatrix} 1 & 1 \\ 0 & 1 \\ \end{bmatrix}$.
\end{lemma}

\begin{proof}
    \uses{Matrix,Matrix.det}
    \leanok
    Brute force.
\end{proof}

For the sake of simplicity of notation, going forward we assume that the submatrix $D_{0}$ in Definition~\ref{standardReprSum3} falls into one of the two special cases presented in Lemma~\ref{Matrix.isUnit_2x2}.

\begin{definition}
    \label{matrix3x3signed}
    \leanok
    We call $D_{0}' \in \mathbb{Q}^{\{x_{1}, x_{0}\} \times \{y_{0}, y_{1}\}}$ the canonical signing of $D_{0} \in \mathbb{Z}_{2}^{\{x_{1}, x_{0}\} \times \{y_{0}, y_{1}\}}$ if
    \[
        D_{0} = \begin{bmatrix}
            1 & 0 \\
            0 & 1 \\
        \end{bmatrix}
        \quad \text{and} \quad
        D_{0}' = \begin{bmatrix}
            1 & 0 \\
            0 & -1
        \end{bmatrix},
        \quad \text{or} \quad
        D_{0} = \begin{bmatrix}
            1 & 1 \\
            0 & 1
        \end{bmatrix}
        \quad \text{and} \quad
        D_{0}' = \begin{bmatrix}
            1 & 1 \\
            0 & 1
        \end{bmatrix}.
    \]
    Similarly, we call $S' \in \mathbb{Q}^{\{x_{2}, x_{1}, x_{0}\} \times \{y_{0}, y_{1}, y_{2}\}}$ the canonical signing of $S \in \mathbb{Z}_{2}^{\{x_{2}, x_{1}, x_{0}\} \times \{y_{0}, y_{1}, y_{2}\}}$ if
    \begin{center}
        \noindent
        \begin{tikzpicture}
            \begin{scope}[scale=0.5, shift={(-3.5, -1.5)}]
                \node[anchor=east] at (0, 1.5) {$S =$};
                \draw (0, 0) -- (3, 0) -- (3, 3) -- (0, 3) -- cycle;
                \draw (0, 2) -- (3, 2);
                \draw (2, 0) -- (2, 3);
                \draw (1, 2) -- (1, 3);
                \draw (2, 1) -- (3, 1);
                \node at (1, 1) {$D_{0}$};
                \node at (0.5, 2.5) {$1$};
                \node at (1.5, 2.5) {$1$};
                \node at (2.5, 2.5) {$0$};
                \node at (2.5, 1.5) {$1$};
                \node at (2.5, 0.5) {$1$};
            \end{scope}
            \node[anchor=west] at (0, 0) {and};
            \begin{scope}[scale=0.5, shift={(3.5, -1.5)}]
                \node[anchor=east] at (0, 1.5) {$S' =$};
                \draw (0, 0) -- (3, 0) -- (3, 3) -- (0, 3) -- cycle;
                \draw (0, 2) -- (3, 2);
                \draw (2, 0) -- (2, 3);
                \draw (1, 2) -- (1, 3);
                \draw (2, 1) -- (3, 1);
                \node at (1, 1) {$D_{0}'$};
                \node at (0.5, 2.5) {$1$};
                \node at (1.5, 2.5) {$1$};
                \node at (2.5, 2.5) {$0$};
                \node at (2.5, 1.5) {$1$};
                \node at (2.5, 0.5) {$1$};
            \end{scope}
        \end{tikzpicture}
    \end{center}
    To simplify notation, going forward we use $D_{0}$, $D_{0}'$, $S$, and $S'$ to refer to the matrices of the form above.
    Observe that the canonical signing $S'$ of $S$ (from Definition~\ref{matrix3x3signed}) is TU.
\end{definition}

\begin{lemma}
    \label{Matrix.HasTuCanonicalSigning.toCanonicalSigning_submatrix3x3}
    \uses{Matrix.IsTuSigningOf,matrix3x3signed}
    \leanok
    Let $Q$ be a TU signing of $S$ (from Definition~\ref{matrix3x3signed}). Let $u \in \{0, \pm 1\}^{\{x_{2}, x_{1}, x_{0}\}}$, $v \in \{0, \pm 1\}^{\{y_{0}, y_{1}, y_{2}\}}$, and $Q'$ be defined as follows:
    \begin{align*}
        u(i) &= \begin{cases}
            Q (x_{2}, y_{0}) \cdot Q (x_{0}, y_{0}), & i = x_{0}, \\
            Q (x_{2}, y_{0}) \cdot Q (x_{0}, y_{0}) \cdot Q (x_{0}, y_{2}) \cdot Q (x_{1}, y_{2}), & i = x_{1}, \\
            1, & i = x_{2}, \\
        \end{cases} \\
        v(j) &= \begin{cases}
            Q (x_{2}, y_{0}), & j = y_{0}, \\
            Q (x_{2}, y_{1}), & j = y_{1}, \\
            Q (x_{2}, y_{0}) \cdot Q (x_{0}, y_{0}) \cdot Q (x_{0}, y_{2}), & j = y_{2}, \\
        \end{cases} \\
        Q' (i, j) &= Q (i, j) \cdot u(i) \cdot v(j) \quad \forall i \in \{x_{2}, x_{1}, x_{0}\}, \ \forall j \in \{y_{0}, y_{1}, y_{2}\}.
    \end{align*}
    Then $Q' = S'$ (from Definition~\ref{matrix3x3signed}).
\end{lemma}

\begin{proof}
    \uses{Matrix.IsTuSigningOf,Matrix.IsTotallyUnimodular.mul_rows,Matrix.IsTotallyUnimodular.mul_cols,Matrix.IsTotallyUnimodular}
    \leanok
    Since $Q$ is a TU signing of $S$ and $Q'$ is obtained from $Q$ by multiplying rows and columns by $\pm 1$ factors, $Q'$ is also a TU signing of $S$. By construction, we have
    \begin{align*}
        Q' (x_{2}, y_{0}) &= Q (x_{2}, y_{0}) \cdot 1 \cdot Q (x_{2}, y_{0}) = 1, \\
        Q' (x_{2}, y_{1}) &= Q (x_{2}, y_{1}) \cdot 1 \cdot Q (x_{2}, y_{1}) = 1, \\
        Q' (x_{2}, y_{2}) &= 0, \\
        Q' (x_{0}, y_{0}) &= Q (x_{0}, y_{0}) \cdot (Q (x_{2}, y_{0}) \cdot Q (x_{0}, y_{0})) \cdot Q (x_{2}, y_{0}) = 1, \\
        Q' (x_{0}, y_{1}) &= Q (x_{0}, y_{1}) \cdot (Q (x_{2}, y_{0}) \cdot Q (x_{0}, y_{0})) \cdot Q (x_{2}, y_{1}), \\
        Q' (x_{0}, y_{2}) &= Q (x_{0}, y_{2}) \cdot (Q (x_{2}, y_{0}) \cdot Q (x_{0}, y_{0})) \cdot (Q (x_{2}, y_{0}) \cdot Q (x_{0}, y_{0}) \cdot Q (x_{0}, y_{2})) = 1, \\
        Q' (x_{1}, y_{0}) &= 0, \\
        Q' (x_{1}, y_{1}) &= Q (x_{1}, y_{1}) \cdot (Q (x_{2}, y_{0}) \cdot Q (x_{0}, y_{0}) \cdot Q (x_{0}, y_{2}) \cdot Q (x_{1}, y_{2})) \cdot (Q (x_{2}, y_{1})), \\
        Q' (x_{1}, y_{2}) &= Q (x_{1}, y_{2}) \cdot (Q (x_{2}, y_{0}) \cdot Q (x_{0}, y_{0}) \cdot Q (x_{0}, y_{2}) \cdot Q (x_{1}, y_{2})) \cdot (Q (x_{2}, y_{0}) \cdot Q (x_{0}, y_{0}) \cdot Q (x_{0}, y_{2})) = 1.
    \end{align*}
    Thus, it remains to show that $Q' (x_{0}, y_{1}) = S' (x_{0}, y_{1})$ and $Q' (x_{1}, y_{1}) = S' (x_{1}, y_{1})$.

    Consider the entry $Q' (x_{0}, y_{1})$. If $D_{0} (x_{0}, y_{1}) = 0$, then $Q' (x_{0}, y_{1}) = 0 = S' (x_{0}, y_{1})$. Otherwise, we have $D_{0} (x_{0}, y_{1}) = 1$, and so $Q' (x_{0}, y_{1}) \in \{\pm 1\}$, as $Q'$ is a signing of $S$. If $Q' (x_{0}, y_{1}) = -1$, then
    \[
        \det Q' (\{x_{0}, x_{2}\}, \{y_{0}, y_{1}\}) = \det \begin{bmatrix} 1 & -1 \\ 1 & 1 \end{bmatrix} = 2 \notin \{0, \pm 1\},
    \]
    which contradicts TUness of $Q'$. Thus, $Q' (x_{0}, y_{1}) = 1 = S' (x_{0}, y_{1})$.

    Consider the entry $Q' (x_{1}, y_{1})$. Since $Q'$ is a signing of $S$, we have $Q' (x_{1}, y_{1}) \in \{\pm 1\}$. Consider two cases.
    \begin{enumerate}
        \item Suppose that $D_{0} = \begin{bmatrix} 1 & 0 \\ 0 & 1 \end{bmatrix}$. If $Q' (x_{1}, y_{1}) = 1$, then
        $
            \det Q = \det \begin{bmatrix}
                1 & 1 & 0 \\
                1 & 0 & 1 \\
                0 & 1 & 1
            \end{bmatrix} = -2 \notin \{0, \pm 1\},
        $
        which contradicts TUness of $Q'$. Thus, $Q' (x_{1}, y_{1}) = -1 = S' (x_{1}, y_{1})$.
        \item Suppose that $D_{0} = \begin{bmatrix} 1 & 1 \\ 0 & 1 \end{bmatrix}$. If $Q' (x_{1}, y_{1}) = -1$, then
        $
            \det Q (\{x_{1}, x_{0}\}, \{y_{1}, y_{2}\}) = \det \begin{bmatrix}
                1 & 1 \\
                -1 & 1
            \end{bmatrix} = 2 \notin \{0, \pm 1\},
        $
        which contradicts TUness of $Q'$. Thus, $Q' (x_{1}, y_{1}) = 1 = S' (x_{1}, y_{1})$.
    \end{enumerate}
\end{proof}

\begin{definition}
    \label{Matrix.toCanonicalSigning}
    \uses{Matrix.IsTotallyUnimodular}
    \leanok
    Let $X$ and $Y$ be sets with $\{x_{2}, x_{1}, x_{0}\} \subseteq X$ and $\{y_{0}, y_{1}, y_{2}\} \subseteq Y$. Let $Q \in \mathbb{Q}^{X \times Y}$ be a TU matrix. Define $u \in \{0, \pm 1\}^{X}$, $v \in \{0, \pm 1\}^{Y}$, and $Q'$ as follows:
    \begin{align*}
        u(i) &= \begin{cases}
            Q (x_{2}, y_{0}) \cdot Q (x_{0}, y_{0}), & i = x_{0}, \\
            Q (x_{2}, y_{0}) \cdot Q (x_{0}, y_{0}) \cdot Q (x_{0}, y_{2}) \cdot Q (x_{1}, y_{2}), & i = x_{1}, \\
            1, & i = x_{2}, \\
            1, & i \in X \setminus \{x_{2}, x_{1}, x_{0}\},
        \end{cases} \\
        v(j) &= \begin{cases}
            Q (x_{2}, y_{0}), & j = y_{0}, \\
            Q (x_{2}, y_{1}), & j = y_{1}, \\
            Q (x_{2}, y_{0}) \cdot Q (x_{0}, y_{0}) \cdot Q (x_{0}, y_{2}), & j = y_{2}, \\
            1, & j \in Y \setminus \{y_{0}, y_{1}, y_{2}\}, \\
        \end{cases} \\
        Q' (i, j) &= Q (i, j) \cdot u(i) \cdot v(j) \quad \forall i \in X, \ \forall j \in Y.
    \end{align*}
    We call $Q'$ the canonical re-signing of $Q$.
\end{definition}

\begin{lemma}
    \label{Matrix.HasTuCanonicalSigning.toCanonicalSigning}
    \uses{Matrix.IsTuSigningOf,matrix3x3signed,Matrix.toCanonicalSigning}
    \leanok
    Let $X$ and $Y$ be sets with $\{x_{2}, x_{1}, x_{0}\} \subseteq X$ and $\{y_{0}, y_{1}, y_{2}\} \subseteq Y$. Let $Q \in \mathbb{Q}^{X \times Y}$ be a TU signing of $Q_{0} \in \mathbb{Z}_{2}^{X \times Y}$ such that $Q_{0} (\{x_{2}, x_{1}, x_{0}\}, \{y_{0}, y_{1}, y_{2}\}) = S$ (from Definition~\ref{matrix3x3signed}). Then the canonical re-signing $Q'$ of $Q$ (from Definition~\ref{Matrix.toCanonicalSigning}) is a TU signing of $Q_{0}$ and $Q' (\{x_{2}, x_{1}, x_{0}\}, \{y_{0}, y_{1}, y_{2}\}) = S'$ (from Definition~\ref{matrix3x3signed}).
\end{lemma}

\begin{proof}
    \uses{Matrix.IsTuSigningOf,Matrix.IsTotallyUnimodular.mul_rows,Matrix.IsTotallyUnimodular.mul_cols,Matrix.HasTuCanonicalSigning.toCanonicalSigning_submatrix3x3}
    \leanok
    Since $Q$ is a TU signing of $Q_{0}$ and $Q'$ is obtained from $Q$ by multiplying some rows and columns by $\pm 1$ factors, $Q'$ is also a TU signing of $Q_{0}$. Equality $Q' (\{x_{2}, x_{1}, x_{0}\}, \{y_{0}, y_{1}, y_{2}\}) = S'$ follows from Lemma~\ref{Matrix.HasTuCanonicalSigning.toCanonicalSigning_submatrix3x3}.
\end{proof}

\begin{definition}
    \label{MatrixSum3.toCanonicalSigning}
    \uses{standardReprSum3,Matrix.IsTuSigningOf,Matrix.toCanonicalSigning}
    \leanok
    Suppose that $B_{\ell}$ and $B_{r}$ from Definition~\ref{standardReprSum3} have TU signings $B_{\ell}'$ and $B_{r}'$, respectively. Let $B_{\ell}''$ and $B_{r}''$ be the canonical re-signings (from Definition~\ref{Matrix.toCanonicalSigning}) of $B_{\ell}'$ and $B_{r}'$, respectively. Let $A_{\ell}''$, $A_{r}''$, $D_{\ell}''$, $D_{r}''$, and $D_{0}''$ be blocks of $B_{\ell}''$ and $B_{r}''$ analogous to blocks $A_{\ell}$, $A_{r}$, $D_{\ell}$, $D_{r}$, and $D_{0}$ of $B_{\ell}$ and $B_{r}$. The canonical signing $B''$ of $B$ is defined as
    \begin{center}
        \noindent
        \begin{tikzpicture}
            \begin{scope}[scale=0.5, shift={(0, -4)}]
                \node[anchor=east] at (0, 4) {$B'' =$};
                \draw (0, 0) -- (8, 0) -- (8, 8) -- (0, 8) -- cycle;
                \draw (0, 4) -- (8, 4);
                \draw (4, 0) -- (4, 8);
                \draw (0, 2) -- (5, 2) -- (5, 5) -- (2, 5) -- (2, 0);
                \draw (3, 4) -- (3, 5);
                \draw (4, 3) -- (5, 3);
                \node at (2, 6) {$A_{\ell}''$};
                \node at (6, 2) {$A_{r}''$};
                \node at (6, 6) {$0$};
                \node at (1, 1) {$D_{\ell r}''$};
                \node at (1, 3) {$D_{\ell}''$};
                \node at (3, 1) {$D_{r}''$};
                \node at (3, 3) {$D_{0}''$};
                \node at (2.5, 4.5) {$1$};
                \node at (3.5, 4.5) {$1$};
                \node at (4.5, 4.5) {$0$};
                \node at (4.5, 3.5) {$1$};
                \node at (4.5, 2.5) {$1$};
            \end{scope}
            \node[anchor=west] at (4.25, 0) {where $D_{\ell r}'' = D_{r}'' \cdot (D_{0}'')^{-1} \cdot D_{\ell}''$.};
        \end{tikzpicture}
    \end{center}
    Note that $D_{0}''$ is non-singular by construction, so $D_{\ell r}''$ and hence $B''$ are well-defined.
\end{definition}

\subsection{Properties of Canonical Signing}

\begin{lemma}
    \label{MatrixSum3.HasCanonicalSigning.toCanonicalSigning}
    \uses{MatrixSum3.toCanonicalSigning}
    \leanok
    $B''$ from Definition~\ref{MatrixSum3.toCanonicalSigning} is a signing of $B$.
\end{lemma}

\begin{proof}
    \uses{Matrix.HasTuCanonicalSigning.toCanonicalSigning,Matrix.IsTuSigningOf}
    \leanok
    By Lemma~\ref{Matrix.HasTuCanonicalSigning.toCanonicalSigning}, $B_{\ell}''$ and $B_{r}''$ are TU signings of $B_{\ell}$ and $B_{r}$, respectively. As a result, blocks $A_{\ell}''$, $A_{r}''$, $D_{\ell}''$, $D_{r}''$, and $D_{0}''$ in $B''$ are signings of the corresponding blocks in $B$. Thus, it remains to show that $D_{\ell r}''$ is a signing of $D_{\ell r}$. This can be verified via a direct calculation.
\end{proof}

\begin{lemma}
    \label{MatrixSum3.HasTuBr.cccAr_isTotallyUnimodular}
    \uses{standardReprSum3,Matrix.IsTuSigningOf,Matrix.toCanonicalSigning}
    \leanok
    Suppose that $B_{r}$ from Definition~\ref{standardReprSum3} has a TU signing $B_{r}'$. Let $B_{r}''$ be the canonical re-signing (from Definition~\ref{Matrix.toCanonicalSigning}) of $B_{r}'$. Let $c_{0}'' = B_{r}'' (X_{r}, y_{0})$, $c_{1}'' = B_{r}'' (X_{r}, y_{1})$, and $c_{2}'' = c_{0}'' - c_{1}''$. Then the following statements hold.
    \begin{enumerate}
        \item\label{item:tss_Brp_c01} For every $i \in X_{r}$, $\begin{bmatrix} c_{0}'' (i) & c_{1}'' (i) \end{bmatrix} \in \{0, \pm 1\}^{\{y_{0}, y_{1}\}} \setminus \left\{ \begin{bmatrix} 1 & -1 \end{bmatrix}, \begin{bmatrix} -1 & 1 \end{bmatrix} \right\}$.
        \item\label{item:tss_Brp_c2} For every $i \in X_{r}$, $c_{2}'' (i) \in \{0, \pm 1\}$.
        \item\label{item:tss_Brp_tu1} $\begin{bmatrix} c_{0}'' & c_{2}'' & A_{r}'' \end{bmatrix}$ is TU.
        \item\label{item:tss_Brp_tu2} $\begin{bmatrix} c_{1}'' & c_{2}'' & A_{r}'' \end{bmatrix}$ is TU.
        \item\label{item:tss_Brp_tu3} $\begin{bmatrix} c_{0}'' & c_{1}'' & c_{2}'' & A_{r}'' \end{bmatrix}$ is TU.
    \end{enumerate}
\end{lemma}

\begin{proof}
    \uses{Matrix.HasTuCanonicalSigning.toCanonicalSigning,Matrix.IsTotallyUnimodular,Matrix.shortTableauPivot,Matrix.IsTotallyUnimodular.shortTableauPivot,Matrix.IsTotallyUnimodular.mul_rows,Matrix.IsTotallyUnimodular.mul_cols}
    \leanok
    Throughout the proof we use that $B_{r}''$ is TU, which holds by Lemma~\ref{Matrix.HasTuCanonicalSigning.toCanonicalSigning}.

    \begin{enumerate}
        \item Since $B_{r}''$ is TU, all its entries are in $\{0, \pm 1\}$, and in particular $\begin{bmatrix} c_{0}'' (i) & c_{1}'' (i) \end{bmatrix} \in \{0, \pm 1\}^{\{y_{0}, y_{1}\}}$. If $\begin{bmatrix} c_{0}' (i) & c_{1}'' (i) \end{bmatrix} = \begin{bmatrix} 1 & -1 \end{bmatrix}$, then
        \[
            \det B_{r}'' (\{x_{2}, i\}, \{y_{0}, y_{1}\}) = \det \begin{bmatrix} 1 & 1 \\ 1 & -1 \end{bmatrix} = -2 \notin \{0, \pm 1\},
        \]
        which contradicts TUness of $B_{r}''$. Similarly, if $\begin{bmatrix} c_{0}'' (i) & c_{1}'' (i) \end{bmatrix} = \begin{bmatrix} -1 & 1 \end{bmatrix} $, then
        \[
            \det B_{r}'' (\{x_{2}, i\}, \{y_{0}, y_{1}\}) = \det \begin{bmatrix} 1 & 1 \\ -1 & 1 \end{bmatrix} = 2 \notin \{0, \pm 1\},
        \]
        which contradicts TUness of $B_{r}''$. Thus, the desired statement holds.

        \item Follows from item~\ref{item:tss_Brp_c01} and a direct calculation.

        \item Performing a short tableau pivot in $B_{r}''$ on $(x_{2}, y_{0})$ yields:
        \[
            B_{r}'' = \begin{bmatrix}
                \fbox{1} & 1 & 0 \\
                c_{0} & c_{1} & A_{r}
            \end{bmatrix}
            \quad \to \quad
            \begin{bmatrix}
                1 & 1 & 0 \\
                -c_{0} & c_{1}'' - c_{0} & A_{r}
            \end{bmatrix}
        \]
        The resulting matrix can be transformed into $\begin{bmatrix} c_{0}'' & c_{2}'' & A_{r}'' \end{bmatrix}$ by removing row $x_{2}$ and multiplying columns $y_{0}$ and $y_{1}$ by $-1$. Since $B_{r}''$ is TU and since TUness is preserved under pivoting, taking submatrices, multiplying columns by ${\pm 1}$ factors, we conclude that $\begin{bmatrix} c_{0}'' & c_{2}'' & A_{r}'' \end{bmatrix}$ is TU.

        \item Similar to item~\ref{item:tss_Brp_tu2}, performing a short tableau pivot in $B_{r}''$ on $(x_{2}, y_{1})$ yields:
        \[
            B_{r}'' = \begin{bmatrix}
                1 & \fbox{1} & 0 \\
                c_{0} & c_{1} & A_{r}
            \end{bmatrix}
            \quad \to \quad
            \begin{bmatrix}
                1 & 1 & 0 \\
                c_{0}'' - c_{1} & -c_{1} & A_{r}
            \end{bmatrix}
        \]
        The resulting matrix can be transformed into $\begin{bmatrix} c_{1}'' & c_{2}'' & A_{r}'' \end{bmatrix}$ by removing row $x_{2}$, multiplying column $y_{1}$ by $-1$, and swapping the order of columns $y_{0}$ and $y_{1}$. Since $B_{r}''$ is TU and since TUness is preserved under pivoting, taking submatrices, multiplying columns by ${\pm 1}$ factors, and re-ordering columns, we conclude that $\begin{bmatrix} c_{1}'' & c_{2}'' & A_{r}'' \end{bmatrix}$ is TU.

        \item Let $V$ be a square submatrix of $\begin{bmatrix} c_{0}'' & c_{1}'' & c_{2}'' & A_{r}'' \end{bmatrix}$. Our goal is to show that $\det V \in \{0, \pm 1\}$.

        Suppose that column $c_{2}''$ is not in $V$. Then $V$ is a submatrix of $B_{r}''$, which is TU. Thus, $\det V \in \{0, \pm 1\}$. Going forward we assume that column $z$ is in $V$.

        Suppose that columns $c_{0}''$ and $c_{1}''$ are both in $V$. Then $V$ contains columns $c_{0}''$, $c_{1}''$, and $c_{2}'' = c_{0}'' - c_{1}''$, which are linearly. Thus, $\det V = 0$. Going forward we assume that at least one of the columns $c_{0}''$ and $c_{1}''$ is not in $V$.

        Suppose that column $c_{1}''$ is not in $V$. Then $V$ is a submatrix of $\begin{bmatrix} c_{0}'' & c_{2}'' & A_{r}'' \end{bmatrix}$, which is TU by item~\ref{item:tss_Brp_tu1}. Thus, $\det V \in \{0, \pm 1\}$. Similarly, if column $c_{0}''$ is not in $V$, then $V$ is a submatrix of $\begin{bmatrix} c_{1}'' & c_{2}'' & A_{r}'' \end{bmatrix}$, which is TU by item~\ref{item:tss_Brp_tu2}. Thus, $\det V \in \{0, \pm 1\}$.

    \end{enumerate}
\end{proof}

\begin{lemma}
    \label{MatrixSum3.HasTuBr.dddAl_isTotallyUnimodular}
    \uses{standardReprSum3,Matrix.IsTuSigningOf,Matrix.toCanonicalSigning}
    \leanok
    Suppose that $B_{\ell}$ from Definition~\ref{standardReprSum3} has a TU signing $B_{\ell}'$. Let $B_{\ell}''$ be the canonical re-signing (from Definition~\ref{Matrix.toCanonicalSigning}) of $B_{\ell}'$. Let $d_{0}'' = B_{\ell}'' (x_{0}, Y_{\ell})$, $d_{1}'' = B_{\ell}'' (x_{1}, Y_{\ell})$, and $d_{2}'' = d_{0}'' - d_{1}''$. Then the following statements hold.
    \begin{enumerate}
        \item\label{item:tss_Blp_d01} For every $j \in Y_{\ell}$, $\begin{bmatrix} d_{0}'' (i) \\ d_{1}'' (j) \end{bmatrix} \in \{0, \pm 1\}^{\{x_{1}, x_{0}\}} \setminus \left\{ \begin{bmatrix} 1 \\ -1 \end{bmatrix}, \begin{bmatrix} -1 \\ 1 \end{bmatrix} \right\}$.
        \item\label{item:tss_Blp_d2} For every $j \in Y_{\ell}$, $d_{2}'' (j) \in \{0, \pm 1\}$.
        \item\label{item:tss_Blp_tu1} $\begin{bmatrix} A_{\ell}'' \\ d_{0}'' \\ d_{2}'' \end{bmatrix}$ is TU.
        \item\label{item:tss_Blp_tu2} $\begin{bmatrix} A_{\ell}'' \\ d_{1}'' \\ d_{2}'' \end{bmatrix}$ is TU.
        \item\label{item:tss_Blp_tu3} $\begin{bmatrix} A_{\ell}'' \\ d_{0}'' \\ d_{1}'' \\ d_{2}'' \end{bmatrix}$ is TU.
    \end{enumerate}
\end{lemma}

\begin{proof}
    \uses{MatrixSum3.HasTuBr.cccAr_isTotallyUnimodular}
    \leanok
    Apply Lemma~\ref{MatrixSum3.HasTuBr.cccAr_isTotallyUnimodular} to $B_{\ell}^{\top}$, or repeat the same arguments up to transposition.
\end{proof}

\begin{lemma}
    \label{MatrixSum3.IsCanonicalSigning.Al_D_isTotallyUnimodular}
    \uses{MatrixSum3.toCanonicalSigning,Matrix.IsTotallyUnimodular}
    \leanok
    Let $B''$ be from Definition~\ref{MatrixSum3.toCanonicalSigning}. Let $c_{0}'' = B'' (X_{r}, y_{0})$, $c_{1}'' = B'' (X_{r}, y_{1})$, and $c_{2}'' = c_{0}'' - c_{1}''$. Similarly, let $d_{0}'' = B'' (x_{0}, Y_{\ell})$, $d_{1}'' = B'' (x_{1}, Y_{\ell})$, and $d_{2}'' = d_{0}'' - d_{1}''$. Then the following statements hold.
    \begin{enumerate}
        \item\label{item:tss_Bp_c2} For every $i \in X_{r}$, $c_{2}'' (i) \in \{0, \pm 1\}$.
        \item\label{item:tss_Bp_Deq} If $D_{0}'' = \begin{bmatrix} 1 & 0 \\ 0 & -1 \end{bmatrix}$, then $D'' = c_{0}'' \otimes d_{0}'' - c_{1}'' \otimes d_{1}''$. If $D_{0}'' = \begin{bmatrix} 1 & 1 \\ 0 & 1 \end{bmatrix}$, then $D'' = c_{0}'' \otimes d_{0}'' - c_{0}'' \otimes d_{1}'' + c_{1}'' \otimes d_{1}''$.
        \item\label{item:tss_Bp_Dcols} For every $j \in Y_{\ell}$, $D'' (X_{r}, j) \in \{0, \pm c_{0}'', \pm c_{1}'', \pm c_{2}''\}$.
        \item\label{item:tss_Bp_Drows} For every $i \in X_{r}$, $D'' (i, Y_{\ell}) \in \{0, \pm d_{0}'', \pm d_{1}'', \pm d_{2}''\}$.
        \item\label{item:tss_Bp_AlD} $\begin{bmatrix} A_{\ell}'' \\ D'' \end{bmatrix}$ is TU.
    \end{enumerate}
\end{lemma}

\begin{proof}
    \uses{MatrixSum3.HasTuBr.cccAr_isTotallyUnimodular,MatrixSum3.HasTuBr.dddAl_isTotallyUnimodular,Matrix.IsTotallyUnimodular}
    \leanok
    \begin{enumerate}
        \item Holds by Lemma~\ref{MatrixSum3.HasTuBr.cccAr_isTotallyUnimodular}.\ref{item:tss_Brp_c2}.

        \item Note that
        \[
            \begin{bmatrix} D_{\ell}'' \\ D_{\ell r}'' \end{bmatrix} = \begin{bmatrix} D_{0}'' \\ D_{r}'' \end{bmatrix} \cdot (D_{0}'')^{-1} \cdot D_{\ell}'', \quad
            \begin{bmatrix} D_{0}'' \\ D_{r}'' \end{bmatrix} = \begin{bmatrix} D_{0}'' \\ D_{r}'' \end{bmatrix} \cdot (D_{0}'')^{-1} \cdot D_{0}'', \quad
            \begin{bmatrix} D_{0}'' \\ D_{r}'' \end{bmatrix} = \begin{bmatrix} c_{0}'' & c_{1}'' \end{bmatrix}, \quad
            \begin{bmatrix} D_{\ell}'' & D_{0}'' \end{bmatrix} = \begin{bmatrix} d_{0}'' \\ d_{1}'' \end{bmatrix}.
        \]
        Thus,
        \[
            D''
            = \begin{bmatrix} D_{\ell}'' & D_{0}'' \\ D_{\ell r}'' & D_{r}'' \end{bmatrix}
            = \begin{bmatrix} D_{0}'' \\ D_{r}'' \end{bmatrix} \cdot (D_{0}'')^{-1} \cdot \begin{bmatrix} D_{\ell}'' & D_{0}'' \end{bmatrix}
            = \begin{bmatrix} c_{0}'' & c_{1}'' \end{bmatrix} \cdot (D_{0}'')^{-1} \cdot \begin{bmatrix} d_{0}'' \\ d_{1}'' \end{bmatrix}.
        \]
        Considering the two cases for $D_{0}''$ and performing the calculations yields the desired results.

        \item Let $j \in Y_{\ell}$. By Lemma~\ref{MatrixSum3.HasTuBr.dddAl_isTotallyUnimodular}.\ref{item:tss_Blp_d01}, $\begin{bmatrix} d_{0}'' (i) \\ d_{1}'' (j) \end{bmatrix} \in \{0, \pm 1\}^{\{x_{1}, x_{0}\}} \setminus \left\{ \begin{bmatrix} 1 \\ -1 \end{bmatrix}, \begin{bmatrix} -1 \\ 1 \end{bmatrix} \right\}$. Consider two cases.
        \begin{enumerate}
            \item If $D_{0}'' = \begin{bmatrix} 1 & 0 \\ 0 & -1 \end{bmatrix}$, then by item~\ref{item:tss_Bp_Deq} we have $D'' (X_{r}, j) = d_{0}'' (j) \cdot c_{0}''  + (-d_{1}'' (j)) \cdot c_{1}''$. By considering all possible cases for $d_{0}'' (j)$ and $d_{1}'' (j)$, we conclude that $D'' (X_{r}, j) \in \{0, \pm c_{0}'', \pm c_{1}'', \pm (c_{0}'' - c_{1}'')\}$.
            \item If $D_{0}'' = \begin{bmatrix} 1 & 1 \\ 0 & 1 \end{bmatrix}$, then by item~\ref{item:tss_Bp_Deq} we have $D'' (X_{r}, j) = (d_{0}'' (j) - d_{1}'' (j)) \cdot c_{0}''  + d_{1}'' (j) \cdot c_{1}''$. By considering all possible cases for $d_{0}'' (j)$ and $d_{1}'' (j)$, we conclude that $D'' (X_{r}, j) \in \{0, \pm c_{0}'', \pm c_{1}'', \pm (c_{0}'' - c_{1}'')\}$.
        \end{enumerate}

        \item Let $i \in X_{r}$. By Lemma~\ref{MatrixSum3.HasTuBr.cccAr_isTotallyUnimodular}.\ref{item:tss_Brp_c01}, $\begin{bmatrix} c_{0}'' (i) & c_{1}'' (i) \end{bmatrix} \in \{0, \pm 1\}^{\{y_{0}, y_{1}\}} \setminus \left\{ \begin{bmatrix} 1 & -1 \end{bmatrix}, \begin{bmatrix} -1 & 1 \end{bmatrix} \right\}$. Consider two cases.
        \begin{enumerate}
            \item If $D_{0}'' = \begin{bmatrix} 1 & 0 \\ 0 & -1 \end{bmatrix}$, then by item~\ref{item:tss_Bp_Deq} we have $D'' (i, Y_{\ell}) = c_{0}'' (i) \cdot d_{0}''  + (-c_{1}'' (i)) \cdot d_{1}''$. By considering all possible cases for $c_{0}'' (i)$ and $c_{1}'' (i)$, we conclude that $D'' (i, Y_{\ell}) \in \{0, \pm d_{0}'', \pm d_{1}'', \pm d_{2}''\}$.
            \item If $D_{0}'' = \begin{bmatrix} 1 & 1 \\ 0 & 1 \end{bmatrix}$, then by item~\ref{item:tss_Bp_Deq} we have $D'' (i, Y_{\ell}) = c_{0}'' (i) \cdot d_{0}'' + (c_{1}'' (i) - c_{0}'' (i)) \cdot d_{1}''$. By considering all possible cases for $c_{0}'' (i)$ and $c_{1}'' (i)$, we conclude that $D'' (i, Y_{\ell}) \in \{0, \pm d_{0}'', \pm d_{1}'', \pm d_{2}''\}$.
        \end{enumerate}


        \item By Lemma~\ref{MatrixSum3.HasTuBr.dddAl_isTotallyUnimodular}.\ref{item:tss_Blp_tu3}, $\begin{bmatrix} A_{\ell}'' \\ d_{0}'' \\ d_{1}'' \\ d_{2}'' \end{bmatrix}$ is TU. Since TUness is preserved under adjoining zero rows, copies of existing rows, and multiplying rows by $\pm 1$ factors, $\begin{bmatrix} A_{\ell}'' \\ 0 \\ \pm d_{0}'' \\ \pm d_{1}'' \\ \pm d_{2}'' \end{bmatrix}$ is also TU. By item~\ref{item:tss_Bp_Drows}, $\begin{bmatrix} A_{\ell}'' \\ D'' \end{bmatrix}$ is a submatrix of the latter matrix, hence it is also TU.
    \end{enumerate}
\end{proof}

\subsection{Proof of Regularity}

\begin{definition}
    \label{MatrixLikeSum3}
    \uses{Matrix.IsTotallyUnimodular}
    \leanok
    Let $X_{\ell}'$, $Y_{\ell}'$, $X_{r}'$, $Y_{r}'$ be sets and let $x_{0}$ and $x_{1}$ be distinct elements contained neither in $X_{\ell}'$ nor $X_{r}'$. Additionally, let $c_{0}, c_{1} \in \mathbb{Q}^{X_{r}' \cup \{x_{1}, x_{0}\}}$ be column vectors. We define $\mathcal{C} (X_{\ell}', Y_{\ell}', X_{r}', Y_{r}'; c_{0}, c_{1})$ to be the family of matrices of the form $\begin{bmatrix} A_{\ell} & 0 \\ D & A_{r} \end{bmatrix}$ such that:
    \begin{enumerate}
        \item $A_{\ell} \in \mathbb{Q}^{X_{\ell}' \times Y_{\ell}'}$, $A_{r} \in \mathbb{Q}^{(X_{r}' \cup \{x_{1}, x_{0}\}) \times Y_{r}'}$, and $D \in \mathbb{Q}^{(X_{r}' \cup \{x_{1}, x_{0}\}) \times Y_{\ell}'}$;
        \item\label{item:mls3_LeftTU} $\begin{bmatrix} A_{\ell} \\ D \end{bmatrix}$ is TU;
        \item\label{item:mls3_Parallels} for every $j \in Y_{\ell}'$, $D (X_{r}', j) \in \{0, \pm c_{0}, \pm c_{1}, \pm (c_{0} - c_{1})\}$;
        \item\label{item:mls3_BottomTU} $\begin{bmatrix} c_{0} & c_{1} & c_{0} - c_{1} & A_{r} \end{bmatrix}$ is TU;
        \item\label{item:mls3_AuxTU} $\begin{bmatrix} A_{\ell} & 0 \\ D (x_{0}, Y_{\ell}') & 1 \\ D (x_{1}, Y_{\ell}') & 1 \end{bmatrix}$ is TU;
        \item\label{item:mls3_Col0} $c_{0} (x_{0}) = 1$ and $c_{0} (x_{1}) = 0$;
        \item\label{item:mls3_Col1} either $c_{1} (x_{0}) = 0$ and $c_{1} (x_{1}) = -1$, or $c_{1} (x_{0}) = 1$ and $c_{1} (x_{1}) = 1$.
    \end{enumerate}
\end{definition}

\begin{lemma}
    \label{MatrixSum3.IsCanonicalSigning.toMatrixLikeSum3}
    \uses{MatrixSum3.toCanonicalSigning,MatrixLikeSum3}
    \leanok
    Let $B''$ be from Definition~\ref{MatrixSum3.toCanonicalSigning}. Then $B'' \in \mathcal{C} (X_{\ell}', Y_{\ell}', X_{r}', Y_{r}'; c_{0}'', c_{1}'')$ with $X_{\ell}' = X_{\ell} \setminus \{x_{1}, x_{0}\}$, $X_{r}' = X_{r} \setminus \{x_{2}, x_{1}, x_{0}\}$, $Y_{\ell}' = Y_{\ell} \setminus \{y_{2}\}$, $Y_{r}' = Y_{r} \setminus \{y_{0}, y_{1}\}$, $x_{0}$ and $x_{1}$ are the same, $c_{0}'' = B'' (X_{r}', y_{0})$, and $c_{1}'' = B'' (X_{r}', y_{1})$.
\end{lemma}

\begin{proof}
    \uses{MatrixSum3.IsCanonicalSigning.Al_D_isTotallyUnimodular,MatrixLikeSum3}
    \leanok
    Recall that $c_{0}'' - c_{1}'' \in \{0, \pm 1\}^{X_{r}'}$ by Lemma~\ref{MatrixSum3.IsCanonicalSigning.Al_D_isTotallyUnimodular}.\ref{item:tss_Bp_c2}, so $\mathcal{C} (X_{\ell}', Y_{\ell}', X_{r}', Y_{r}'; c_{0}'', c_{1}'')$ is well-defined. To see that $B'' \in \mathcal{C} (X_{\ell}', Y_{\ell}', X_{r}', Y_{r}'; c_{0}'', c_{1}'')$, note that all properties from Definition~\ref{MatrixLikeSum3} are satisfied: property~\ref{item:mls3_Parallels} holds by Lemma~\ref{MatrixSum3.IsCanonicalSigning.Al_D_isTotallyUnimodular}.\ref{item:tss_Bp_Dcols}, property~\ref{item:mls3_BottomTU} holds by Lemma~\ref{MatrixSum3.HasTuBr.cccAr_isTotallyUnimodular}.\ref{item:tss_Brp_tu3}, and property~\ref{item:mls3_LeftTU} holds by Lemma~\ref{MatrixSum3.IsCanonicalSigning.Al_D_isTotallyUnimodular}.\ref{item:tss_Bp_AlD}.
\end{proof}

\begin{lemma}
    \label{MatrixLikeSum3.shortTableauPivot}
    \uses{MatrixLikeSum3,Matrix.shortTableauPivot}
    \leanok
    Let $C \in \mathcal{C} (X_{\ell}', Y_{\ell}', X_{r}', Y_{r}'; c_{0}, c_{1})$ from Definition~\ref{MatrixLikeSum3}. Let $x \in X_{\ell}'$ and $y \in Y_{\ell}'$ be such that $A_{\ell} (x, y) \neq 0$, and let $C'$ be the result of performing a short tableau pivot in $C$ on $(x, y)$. Then $C' \in \mathcal{C} (X_{\ell}', Y_{\ell}', X_{r}', Y_{r}'; c_{0}, c_{1})$.
\end{lemma}

\begin{proof}
    \uses{Matrix.shortTableauPivot_zero,Matrix.IsTotallyUnimodular,Matrix.IsTotallyUnimodular.shortTableauPivot}
    \leanok
    Our goal is to show that $C'$ satisfies all properties from Definition~\ref{MatrixLikeSum3}. Let $C' = \begin{bmatrix} C_{11}' & C_{12}' \\ C_{21}' & C_{22}' \end{bmatrix}$, and let $\begin{bmatrix} A_{\ell}' \\ D' \end{bmatrix}$ be the result of performing a short tableau pivot on $(x, y)$ in $\begin{bmatrix} A_{\ell} \\ D \end{bmatrix}$. Observe the following.

    \begin{itemize}
        \item By Lemma~\ref{Matrix.shortTableauPivot_zero}, $C_{11}' = A_{\ell}'$, $C_{12}' = 0$, $C_{21}' = D'$, and $C_{22}' = A_{r}$.
        \item Since $\begin{bmatrix} A_{\ell} \\ D \end{bmatrix}$ is TU by property~\ref{item:mls3_LeftTU} for $C$, all entries of $A_{\ell}$ are in $\{0, \pm 1\}$.
        \item $A_{\ell} (x, y) \in \{\pm 1\}$, as $A_{\ell} (x, y) \in \{0, \pm 1\}$ by the above observation and $A_{\ell} (x, y) \neq 0$ by the assumption.
        \item Since $\begin{bmatrix} A_{\ell} \\ D \end{bmatrix}$ is TU by property~\ref{item:mls3_LeftTU} for $C$, and since pivoting preserves TUness, $\begin{bmatrix} A_{\ell}' \\ D' \end{bmatrix}$ is also TU.
    \end{itemize}

    These observations immediately imply properties~\ref{item:mls3_BottomTU} and~\ref{item:mls3_LeftTU} for $C'$. Indeed, property~\ref{item:mls3_BottomTU} holds for $C'$, since $C_{22}' = A_{r}$ and $\begin{bmatrix} c_{0} & c_{1} & c_{0} - c_{1} & A_{r} \end{bmatrix}$ is TU by property~\ref{item:mls3_BottomTU} for $C$. On the other hand, property~\ref{item:mls3_LeftTU} follows from $C_{11}' = A_{\ell}'$, $C_{21}' = D'$, and $\begin{bmatrix} A_{\ell}' \\ D' \end{bmatrix}$ being TU. Thus, it only remains to show that $C'$ satisfies property~\ref{item:mls3_Parallels}. Let $j \in Y_{r}$. Our goal is to prove that $D' (X_{r}', j) \in \{0, \pm c_{0}, \pm c_{1}, \pm (c_{0} - c_{1})\}$.

    Suppose $j = y$. By the pivot formula, $D' (X_{r}', y) = -\frac{D (X_{r}', y)}{A_{\ell} (x, y)}$. Since $D (X_{r}', y) \in \{0, \pm c_{0}, \pm c_{1}, \pm (c_{0} - c_{1})\}$ by property~\ref{item:mls3_Parallels} for $C$ and since $A_{\ell} (x, y) \in \{\pm 1\}$, we get $D' (X_{r}', y) \in \{0, \pm c_{0}, \pm c_{1}, \pm (c_{0} - c_{1})\}$.

    Now suppose $j \in Y_{\ell} \setminus \{y\}$. By the pivot formula, $D' (X_{r}', j) = D (X_{r}', j) - \frac{A_{\ell} (x, j)}{A_{\ell} (x, y)} \cdot D (X_{r}', y)$. Here $D (X_{r}', j), \ D (X_{r}', y) \in \{0, \pm c_{0}, \pm c_{1}, \pm (c_{0} - c_{1})\}$ by property~\ref{item:mls3_Parallels} for $C$, and $A_{\ell} (x, j) \in \{0, \pm 1\}$ and $A_{\ell} (x, y) \in \{\pm 1\}$ by the prior observations. Perform an exhaustive case distinction on $D (X_{r}', j)$, $D (X_{r}', y)$, $A_{\ell} (x, j)$, and $A_{\ell} (x, y)$. The number of cases can be significantly reduced by using symmetries. In every remaining case, we can either show that $D' (X_{r}', j) \in \{0, \pm c_{0}, \pm c_{1}, \pm (c_{0} - c_{1})\}$, as desired, or obtain a contradiction with property~\ref{item:mls3_AuxTU}.
\end{proof}

\begin{lemma}
    \label{MatrixLikeSum3.isTotallyUnimodular}
    \uses{MatrixLikeSum3,Matrix.IsTotallyUnimodular}
    \leanok
    Let $C \in \mathcal{C} (X_{\ell}', Y_{\ell}', X_{r}', Y_{r}'; c_{0}, c_{1})$ from Definition~\ref{MatrixLikeSum3}. Then $C$ is TU.
\end{lemma}

\begin{proof}
    \uses{MatrixLikeSum3,Matrix.isTotallyUnimodular_iff_forall_isPartiallyUnimodular,shortTableauPivot_submatrix_det_abs_eq_div,MatrixLikeSum3.shortTableauPivot}
    \leanok
    By Lemma~\ref{Matrix.isTotallyUnimodular_iff_forall_isPartiallyUnimodular}, it suffices to show that $C$ is $k$-PU for every $k \in \mathbb{N}$. We prove this claim by induction on $k$. The base case with $k = 1$ holds, since properties~\ref{item:mls3_BottomTU} and~\ref{item:mls3_LeftTU} in Definition~\ref{MatrixLikeSum3} imply that $A_{\ell}$, $A_{r}$, and $D$ are TU, so all their entries of $C = \begin{bmatrix} A_{\ell} & 0 \\ D & A_{r} \end{bmatrix}$ are in $\{0, \pm 1\}$, as desired.

    Suppose that for some $k \in \mathbb{N}$ we know that every $C' \in \mathcal{C} (X_{\ell}', Y_{\ell}', X_{r}', Y_{r}'; c_{0}, c_{1})$ is $k$-PU. Our goal is to show that $C$ is $(k + 1)$-PU, i.e., that every $(k + 1) \times (k + 1)$ submatrix $S$ of $C$ has $\det V \in \{0, \pm 1\}$.

    First, suppose that $V$ has no rows in $X_{\ell}'$. Then $V$ is a submatrix of $\begin{bmatrix} D & A_{r} \end{bmatrix}$, which is TU by property~\ref{item:mls3_BottomTU} in Definition~\ref{MatrixLikeSum3}, so $\det V \in \{0, \pm 1\}$. Thus, we may assume that $S$ contains a row $x_{\ell} \in X_{\ell}'$.

    Next, note that without loss of generality we may assume that there exists $y_{\ell} \in Y_{\ell}'$ such that $V (x_{\ell}, y_{\ell}) \neq 0$. Indeed, if $V (x_{\ell}, y) = 0$ for all $y$, then $\det V = 0$ and we are done, and $V (x_{\ell}, y) = 0$ holds whenever $y \in Y_{r}'$.

    Since $C$ is $1$-PU, all entries of $V$ are in $\{0, \pm 1\}$, and hence $V (x_{\ell}, y_{\ell}) \in \{\pm 1\}$. Thus, by Lemma~\ref{shortTableauPivot_submatrix_det_abs_eq_div}, performing a short tableau pivot in $V$ on $(x_{\ell}, y_{\ell})$ yields a matrix that contains a $k \times k$ submatrix $S''$ such that $|\det V| = |\det V''|$. Since $V$ is a submatrix of $C$, matrix $V''$ is a submatrix of the matrix $C'$ resulting from performing a short tableau pivot in $C$ on the same entry $(x_{\ell}, y_{\ell})$. By Lemma~\ref{MatrixLikeSum3.shortTableauPivot}, we have $C' \in \mathcal{C} (X_{\ell}', Y_{\ell}', X_{r}', Y_{r}'; c_{0}, c_{1})$. Thus, by the inductive hypothesis applied to $V''$ and $C'$, we have $\det V'' \in \{0, \pm 1\}$. Since $|\det V| = |\det V''|$, we conclude that $\det V \in \{0, \pm 1\}$.
\end{proof}

\begin{lemma}
    \label{MatrixSum3.IsCanonicalSigning.isTotallyUnimodular}
    \uses{MatrixSum3.toCanonicalSigning,Matrix.IsTotallyUnimodular}
    \leanok
    $B''$ from Definition~\ref{MatrixSum3.toCanonicalSigning} is TU.
\end{lemma}

\begin{proof}
    \uses{MatrixSum3.IsCanonicalSigning.toMatrixLikeSum3,MatrixLikeSum3.isTotallyUnimodular}
    \leanok
    Combine the results of Lemmas~\ref{MatrixSum3.IsCanonicalSigning.toMatrixLikeSum3} and~\ref{MatrixLikeSum3.isTotallyUnimodular}.
\end{proof}

\begin{theorem}
    \label{Matroid.IsSum3of.isRegular}
    \uses{Matroid.IsSum3of,Matroid.IsRegular}
    \leanok
    Let $M$ be a $3$-sum of regular matroids $M_{\ell}$ and $M_{r}$. Then $M$ is also regular.
\end{theorem}

\begin{proof}
    \uses{StandardRepr,Matroid.IsSum3of,StandardRepr.toMatroid_isRegular_iff_hasTuSigning,MatrixSum3.toCanonicalSigning,MatrixSum3.HasCanonicalSigning.toCanonicalSigning,MatrixSum3.IsCanonicalSigning.isTotallyUnimodular,Matrix.isUnit_2x2}
    \leanok
    Let $B_{\ell}$, $B_{r}$, and $B$ be standard $\mathbb{Z}_{2}$ representation matrices from Definition~\ref{Matroid.IsSum3of}. Since $M_{\ell}$ and $M_{r}$ are regular, by Lemma~\ref{StandardRepr.toMatroid_isRegular_iff_hasTuSigning}, $B_{\ell}$ and $B_{r}$ have TU signings. Then the canonical signing $B''$ from Definition~\ref{MatrixSum3.toCanonicalSigning} is a TU signing of $B$. Indeed, $B''$ is a signing of $B$ by Lemma~\ref{MatrixSum3.HasCanonicalSigning.toCanonicalSigning}, and $B''$ is TU by Lemma~\ref{MatrixSum3.IsCanonicalSigning.isTotallyUnimodular}. Thus, $M$ is regular by Lemma~\ref{StandardRepr.toMatroid_isRegular_iff_hasTuSigning}.
\end{proof}

%% file: special.tex
\section{Special Matroids}

\begin{definition}
    \label{Matrix.IsGraphic}
    \uses{Matrix}
    \leanok
    Let $A \in \mathbb{Q}^{X \times Y}$ be a matrix. If for all $j \in Y$, one has that $a_{i,j} = 0$ for all $i \in X$, or that there exists $i_1,i_2 \in X$ such that
    \[
    a_{i,j} = \begin{cases}
        1 & \text{ if $i = i_1$} \\
        -1 & \text{ if $i = i_2$} \\
        0 & \text{ otherwise},
    \end{cases}
    \]
    then we call $A$ a node-incidence matrix for a (directed) graph whose nodes are indexed by $X$ and whose edges are indexed by $Y$.
\end{definition}

\begin{definition}
    \label{Matroid.IsGraphic}
    \uses{VectorMatroid,Matrix.IsGraphic}
    \leanok
    We say that a matroid is graphic if it can be represented by a node-incidence matrix.
\end{definition}

\begin{definition}
    \label{StandardRepr.dual}
    \uses{StandardRepr}
    \leanok
    Let $S$ be a standard representation given by matrix $B$. The dual of $S$ is given by $-B^\intercal$.
\end{definition}

\begin{definition}
    \label{Matroid.IsCographic}
    \uses{Matroid.IsGraphic,StandardRepr.dual}
    \leanok
    We say a matroid is co-graphic if its dual is graphic.
\end{definition}

\begin{definition}
    \label{matroidR10}
    \uses{StandardRepr}
    \leanok
    The matroid with standard representation
        \[\begin{bmatrix}
            1 & 0 & 0 & 1 & 1 \\
            1 & 1 & 0 & 0 & 1 \\
            0 & 1 & 1 & 0 & 1 \\
            0 & 0 & 1 & 1 & 1 \\
            1 & 1 & 1 & 1 & 1 \\
        \end{bmatrix}\]
        over $\mathbb{Z}_2$ is called $R_{10}.$
    \end{definition}

\begin{theorem}
    \label{matroidR10.isRegular}
    \uses{matroidR10,Matroid.IsRegular}
    \leanok
    The matroid $R_{10}$ is regular.
\end{theorem}

\begin{proof}
    \leanok
    See Lean implementation.
\end{proof}

\begin{theorem}
    \label{Matroid.IsGraphic.isRegular}
    \uses{Matroid.IsGraphic,Matroid.IsRegular}
    \leanok
    Every graphic matroid is regular.
\end{theorem}

\begin{proof}
    \leanok
    See Lean implementation.
\end{proof}

%% file: cographic.tex
\subsection{Regularity of Cographic Matroids}

\begin{lemma}[Row space of a standard representation]
    \label{lem:row-space-standard}
    Let $X$ and $Y$ be disjoint finite sets and let
    
    \[
      B \in \mathbb{F}_2^{X \times Y}.
    \]

    Consider the matrix
    
    \[ A := [\, \mathbf{1}_x \mid B \,] \in \mathbb{F}_2^{X \times (X \cup Y)},\]
    
    where the columns are indexed by $E := X \cup Y$ and the rows by $X$.
    Then the row space of $A$ is
    
    \[ \text{row}(A) \;=\; \{\, (u,\, uB) \mid u \in \mathbb{F}_2^X \,\} \;\subseteq\; \mathbb{F}_2^X \oplus \mathbb{F}_2^Y \cong \mathbb{F}_2^{E}. \]
\end{lemma}

\begin{proof}
    The $x$-th row of $A$ is $(e_x, B_{x,*})$, where $e_x$ is the standard basis vector
    in $\mathbb{F}_2^X$ and $B_{x,*}$ is the $x$-th row of $B$.
    A general linear combination of the rows is therefore
    \[
      \sum_{x \in X} u_x (e_x, B_{x,*})
      = \bigl( u,\, \sum_{x \in X} u_x B_{x,*} \bigr)
      = (u,\, uB),
    \]
    where $u = (u_x)_{x \in X} \in \mathbb{F}_2^X$.
    Conversely, every pair $(u, uB)$ arises in this way, so these are exactly the row vectors.
\end{proof}

\begin{lemma}[Orthogonal complement of a standard row space]
    \label{lem:orth-complement-standard}
    Let $A = [\mathbf{1}_x \mid B]$ be as in Lemma~\ref{lem:row-space-standard}, and let
    \[U := \text{row}(A) \subseteq \mathbb{F}_2^{X \cup Y}.\]
    Then the orthogonal complement of $U$ is
    \[U^\perp = \{\, (b B^{\mathsf T},\, b) \mid b \in \mathbb{F}_2^Y \,\}.\]
    Equivalently, if $B^* := -B^{\mathsf T}$, then
    \[U^\perp= \{\, (b B^*,\, b) \mid b \in \mathbb{F}_2^Y \,\}.
    \]
\end{lemma}

\begin{proof}
    Write vectors in $\mathbb{F}_2^{X \cup Y}$ as pairs $(a,b)$ with
    $a \in \mathbb{F}_2^X$ and $b \in \mathbb{F}_2^Y$.
    By Lemma~\ref{lem:row-space-standard}, any element of $U$ has the form
    $(u, uB)$ with $u \in \mathbb{F}_2^X$.
    The orthogonality condition $(a,b) \in U^\perp$ means
    \begin{align*}
      0 &= (a,b) \cdot (u, uB) \\ 
        &= a \cdot u + b \cdot (uB) \\
        &= a \cdot u + (b B^{\mathsf T}) \cdot u \\
        &= (a + b B^{\mathsf T}) \cdot u
    \end{align*}
    for all $u \in \mathbb{F}_2^X$.
    Hence we must have $a = b B^{\mathsf T}$, and then
    \[
      U^\perp = \{\, (b B^{\mathsf T}, b) \mid b \in \mathbb{F}_2^Y \,\}.
    \]
    Over $\mathbb{F}_2$ we have $-1 = 1$, so $B^* = -B^{\mathsf T} = B^{\mathsf T}$,
    yielding the alternative description.
\end{proof}

\begin{lemma}[Row space of the dual standard matrix]
    \label{lem:row-space-dual-standard}
    With $B$ and $B^* = -B^{\mathsf T}$ as above, define
    \[A^* := [\, \mathbf{1}_y \mid B^* \,] \in \mathbb{F}_2^{Y \times (X \cup Y)}.\]
    Then
    \[\text{row}(A^*) = U^\perp,\]
    where $U = \text{row}(A)$ and $U^\perp$ is given by
    Lemma~\ref{lem:orth-complement-standard}.
\end{lemma}

\begin{proof}
    The $y$-th row of $A^*$ is $(e_y, B^*_{y,*})$ with $e_y \in \mathbb{F}_2^Y$.
    A general linear combination of the rows is
    \[\sum_{y \in Y} b_y (e_y, B^*_{y,*}) = \bigl( b,\, b B^* \bigr),
    \]
    where $b = (b_y)_{y \in Y} \in \mathbb{F}_2^Y$.
    Thus
    \[\text{row}(A^*) = \{\, (b, b B^*) \mid b \in \mathbb{F}_2^Y \,\}.\]
    Identifying $\mathbb{F}_2^{X \cup Y}$ as $\mathbb{F}_2^X \oplus \mathbb{F}_2^Y$
    with coordinates ordered as $(X, Y)$, this is exactly the set
    \[\{\, (b B^*, b) \mid b \in \mathbb{F}_2^Y \,\},\]
    which coincides with $U^\perp$ by Lemma~\ref{lem:orth-complement-standard}.
\end{proof}

\begin{lemma}[Dual vector matroid via orthogonal complement]
    \label{lem:dual-vectormatroid-orth}
    Let $A$ and $A'$ be matrices over a field $F$ with the same column index set $E$,
    and suppose
    \[\text{row}(A') = \text{row}(A)^{\perp} \subseteq F^{E}.\]
    Let $M(A)$ and $M(A')$ be the vector matroids represented by $A$ and $A'$. Then
    \[M(A') = M(A)^{*}.\]
\end{lemma}

\begin{proof}
    Let $F \subseteq E$.  

    \emph{($\Rightarrow$)}  
    Suppose $F$ is dependent in $M(A)$.  
    Then there exists a nonzero vector $c \in F^{F}$ such that $A_F c = 0$.  
    Extend $c$ by zero outside $F$ (still denoted $c$).  
    The condition $A c = 0$ means each row $r$ of $A$ satisfies $r \cdot c = 0$,  
    hence $c \in \text{row}(A)^{\perp} = \text{row}(A')$.  
    Write
    \[
        c = \sum_i \lambda_i r'_i,
    \]
    where the $r'_i$ are rows of $A'$ and not all $\lambda_i$ are zero.  
    For every $e \in E\setminus F$ we have $c_e = 0$, so
    \[
        \left(\sum_i \lambda_i r'_i\right)\big|_{E\setminus F} = 0.
    \]
    Hence the rows of $A'$ indexed by $E \setminus F$ admit a nontrivial
    linear combination giving the zero row, so $E \setminus F$ is dependent in $M(A')$.

    \emph{($\Leftarrow$)}  
    The same argument with $A$ and $A'$ interchanged, using
    $\text{row}(A) = (\text{row}(A')^\perp)$, shows that if
    $E\setminus F$ is dependent in $M(A')$, then $F$ is dependent in $M(A)$.

    Thus
    \[
        F \text{ dependent in } M(A)
        \;\Longleftrightarrow\;
        E \setminus F \text{ dependent in } M(A'),
    \]
    which is the defining property of duality.
\end{proof}

\begin{theorem}[Dual of standard representation corresponds to dual matroid]
\label{thm:StandardRepr.toMatroid_dual}
Let M be a binary matroid on ground set $E = X \cup Y$, with
standard representation $B$ so that
\[A = [\, \mathbf{1}_X \mid B \,].\]
Let $B^{*} := -B^{\mathsf T}$ and
\[A^{*} := [\, \mathbf{1}_Y \mid B^{*} \,].\]
Then $M(A^{*}) = M(A)^{*} = M^{*}$.
\end{theorem}

\begin{proof}
    By Lemma~\ref{lem:row-space-standard} and Lemma~\ref{lem:orth-complement-standard},
    if $U = \text{row}(A)$ then $U^\perp$ has the form
    \[U^\perp = \{\, (b B^*, b) \mid b \in \mathbb{F}_2^Y \,\}.\]
    By Lemma~\ref{lem:row-space-dual-standard}, we have
    \[\text{row}(A^*) = U^\perp = \text{row}(A)^\perp.\]
    Therefore, by Lemma~\ref{lem:dual-vectormatroid-orth},
    the column-matroid $M(A^*)$ is the dual of $M(A)$:
    \[M(A^*) = M(A)^* = M^*.\]
\end{proof}

\begin{lemma}
    \label{lem:Matroid.Dual.isRegular}
    The dual matroid of a regular matroid is also a regular matroid.
\end{lemma}

\begin{proof}
    Let $M$ be a regular matroid. We wish to show that $M^{*}$ is also regular. 
    
    Take a standard $ \mathbb Z_2$-representation matrix $B$ of $M$. By Lemma~\ref{StandardRepr.toMatroid_isRegular_iff_hasTuSigning}, since $M$ is regular, there exists a TU signing $B'$ of $B$: $B'$ is a matrix over $\mathbb Q$ that is TU, and $|B'(i,j)| = B(i,j)$ for all entries.  So $M$ is represented (over $\mathbb Q$) by a TU matrix $B'$ whose pattern of zero and non-zero entries is exactly that of $B$.

    From Theorem \ref{thm:StandardRepr.toMatroid_dual}, if a matroid $M$ has standard representation matrix $B$, then its dual $M^{*}$ has the standard representation matrix $B^{*} = -B^{\intercal}$. The TU signing of this dual standard matrix, $(B')^{*} = -(B')^{\intercal}$, preserves total unimodularity, so $(B')^{*}$ is a TU matrix whose support is exactly $B^{*}$.

    Since we have just exhibited a TU signing of $M^{*}$ (i.e., $(B')^{*}$), the dual matroid $M^{*}$ is regular by Lemma~\ref{StandardRepr.toMatroid_isRegular_iff_hasTuSigning}.
\end{proof}

\begin{theorem}
    \label{Matroid.IsCographic.isRegular}
    Every cographic matroid is regular.
\end{theorem}

\begin{proof}
    We know that all graphic matroids are regular by Theorem~\ref{Matroid.IsGraphic.isRegular}. Recall that we say a matroid is cographic if its dual is graphic. So it suffices to show regularity is preserved under duals, which we showed in Lemma \ref{lem:Matroid.Dual.isRegular}.
\end{proof}

%% file: conclusion.tex
\section{Conclusion}

\begin{definition}
    \label{Matroid.IsGood}
    \uses{Matroid.IsSum1of,Matroid.IsSum2of,Matroid.IsSum3of,Matroid.IsGraphic,Matroid.IsCographic,matroidR10}
    \leanok
    Any graphic matroid is good.
    Any cographic matroid is good.
    Any matroid isomorphic to R10 is good.
    Any 1-sum (in the sense of Definition~\ref{Matroid.IsSum1of}) of good matroids is a good matroid.
    Any 2-sum (in the sense of Definition~\ref{Matroid.IsSum2of}) of good matroids is a good matroid.
    Any 3-sum (in the sense of Definition~\ref{Matroid.IsSum3of}) of good matroids is a good matroid.
\end{definition}

\begin{corollary}
    \label{Matroid.IsGood.isRegular}
    \uses{Matroid.IsGood,Matroid.IsRegular}
    \leanok
    Any good matroid is regular. This is a corollary of the easy direction of the Seymour theorem.
\end{corollary}

\begin{proof}
    \uses{Matroid.IsSum1of.isRegular,Matroid.IsSum2of.isRegular,Matroid.IsSum3of.isRegular,Matroid.IsGraphic.isRegular,Matroid.IsCographic.isRegular,matroidR10.isRegular}
    Structural induction
using theorems \ref{Matroid.IsGraphic.isRegular},
\ref{Matroid.IsCographic.isRegular},
\ref{matroidR10.isRegular},
\ref{Matroid.IsSum1of.isRegular},
\ref{Matroid.IsSum2of.isRegular}, and
\ref{Matroid.IsSum3of.isRegular}.
\end{proof}